\pdfoutput=1
\documentclass[a4paper,12pt]{article}
\usepackage{graphicx}
\usepackage[T1]{fontenc}
\usepackage[utf8]{inputenc}
\usepackage[english]{babel}
\usepackage{amsmath,amsthm,amssymb,amsfonts}
\usepackage{subcaption}
\usepackage{geometry}
\usepackage{cite}
\usepackage{algorithmic} 
\usepackage{algorithm}
\usepackage{todonotes}
\usepackage[hidelinks]{hyperref}

\newtheorem{theorem}{Theorem}
\theoremstyle{definition}
\newtheorem{lemma}{Lemma} 
\newtheorem{remark}{Remark}

\newcommand{\n}[1]{\left\|#1 \right\|} 
\renewcommand{\a}{\alpha}
\renewcommand{\b}{\beta}

\renewcommand{\d}{\delta}
\newcommand{\D}{\Delta}
\newcommand{\la}{\lambda}
\renewcommand{\t}{\tau}

\renewcommand{\th}{\theta}

\newcommand{\x}{\bar x}

\newcommand{\R}{\mathbb R}
\newcommand{\E}{\mathcal E}
\newcommand{\N}{\mathbb N}
\newcommand{\Z}{\mathbb Z}
\newcommand{\g}{\geq}
\renewcommand{\l}{\leq}

\newcommand{\lr}[1]{\left\langle #1\right\rangle} 
\newcommand{\nf}[1]{\nabla f(#1)}
\DeclareMathOperator{\prox}{prox}
\DeclareMathOperator{\argmin}{argmin}
\DeclareMathOperator{\dom}{dom}
\DeclareMathOperator{\id}{Id}
\DeclareMathOperator{\conv}{conv}

\begin{document}
\title{Proximal extrapolated gradient methods for variational inequalities}
\author{Yu. Malitsky\thanks{Institute for Computer Graphics and Vision, Graz University of Technology, 8010 Graz,
Austria. e-mail: yura.malitsky@icg.tugraz.at}}

\maketitle
\begin{abstract}
    The paper concerns with novel first-order methods for monotone
    variational inequalities. They use a very simple linesearch
    procedure that takes into account a local information of the
    operator. Also the methods do not require Lipschitz-continuity of
    the operator and the linesearch procedure uses only values of the
    operator. Moreover, when operator is affine our linesearch becomes
    very simple, namely, it needs only vector-vector multiplication. For all
    our methods we establish the ergodic convergence rate.  Although
    the proposed methods are very general, sometimes they may show
    much better performance even for optimization problems. The reason
    for this is that they often can use larger stepsizes without
    additional expensive computation. 
 \end{abstract}

{\small
\noindent
{\bfseries 2010 Mathematics Subject Classification:}
{47J20, 65K10, 	65K15, 65Y20, 90C33}

\noindent {\bfseries Keywords:}
variational inequality,  monotone operator,  linesearch, nonmonotone
stepsizes,\\ proximal methods, convex optimization, ergodic convergence.
}

\section{Introduction}
This paper considers a problem of the variational
inequality in a general form
\begin{equation}
    \label{gvi}
\text{find } x^*\in \E \colon \quad     \lr{F(x^*), x - x^* } + g(x) - g(x^*) \geq 0 \quad \forall x \in \E,
\end{equation}
where  $F\colon \E \to \E$ is a monotone operator and
$g\colon \E\to (-\infty, +\infty]$ is a convex function. This is an
important problem that has a variety of theoretical and practical
applications \cite{pang:03,konnov07,ekeland:76}.

The main iteration step of the proposed method is defined as follows
\begin{align*} 
    y_n & =  x_n + \t_n (x_n - x_{n-1}) \\
    x_{n+1} & = \prox_{\la_n g} (x_n - \la_n F(y_n)),
\end{align*}
where we define $\t_n$, $\la_n$, and $y_n$ from local properties of
$F(y_n)$. For this in each iteration we run some simple linesearch
procedure. We propose different procedures for different cases: for a
general problem \eqref{gvi}, for \eqref{gvi} with $g(x)=\d_C(x)$, and
for a case when $F$ is a gradient of a convex differentiable function.
Each iteration of the linesearch procedure requires only one value of
$F$ and function $g$ is not used at all. In contrast to many known
methods, we do not require monotonicity of stepsizes $(\la_n)$.  Also
in case when $F$ is affine our linesearch procedures needs only
vector-vector computation. Moreover, our analysis does not need a
Lipschitz assumption on $F$, only locally Lipschitz one.

Although we consider quite a general problem, our discussion presented
below consists from two separate parts devoted to the optimization
problems and variational inequality problems.  This is because we
noticed that for some difficult optimization problems our algorithm
may work much better than some existing methods.
 Next section after the introduction devotes to studying of our two methods.
We show their globally convergence, consider some particular cases and
establish complexity rates. In Section~\ref{sec:alg3} we consider a
problem of composite minimization for which we improve one of our
methods. In Section~\ref{sec:comparison} we study some known linesearch
procedures and make numerical illustrations of our methods with
several popular methods.

\subsection{Preliminaries}\label{prelim}
In what follows, $\E$ denotes a finite-dimensional real vector space
with inner product $\lr{\cdot,\cdot}$ and norm $\n{\cdot}$, $\nabla f
$ denotes a gradient of a smooth function $f$.  For a proper lower
semicontinuous convex function $g\colon \E\to (-\infty, \infty]$ we
denote its domain by $\dom g$, i.e., $\dom g:= \{x\in \E\colon g(x) <
\infty \}$.  The proximal operator $\prox_g\colon \E\to \E$ is defined
as
\begin{equation*}
    \prox_g(y) := \argmin_{x\in \E}\bigl\{g(x) + \frac 1 2 \n{x-y}^2\bigr\}.
\end{equation*}
For a set $C$, we denote by $\delta_C$ the indicator function of the
set, that is, $\delta_C(x) = 0$ if $x \in C$ and $\infty$
otherwise. A metric projector onto $C$ we denote as $P_C$. Clearly, by
definition, $P_C x = \prox_{\d_C}x$.

The operator $F$ is called monotone if $$\lr{F(x)-F(y),x-y}\geq
0\quad \forall x,y \in \E.$$

\subsection{Optimization perspective}
Consider the following problem of composite minimization
\begin{equation}
    \label{min}
    \min_x \Phi(x):= f(x) + g(x),
\end{equation}
where $f$ is a differentiable convex function, and $g$ is a proper
lower semicontinuous convex function.  Such formulation assumes that
we know the structure of the underlying function $\Phi$. It is not
difficult to verify that the first order optimality conditions of
\eqref{min} are a particular case of \eqref{gvi} with $F = \nabla{f}$.

Problem
\eqref{min} is rich enough to encompass many important applications in machine learning, image
processing, compressed sensing, statistics, etc.
\cite{nesterov07,boyd:proximal,combettes:11,tseng08,fista,cevher:14,combettes08}.
Although first-order methods for problem~\eqref{min} have a long history, they continue to receive
much attention from optimization community. Many real-life applications are large-scale and in this
case first-order methods often outperform other methods such as interior point methods, Newton methods,
since the iterations of the former are much cheaper and do not depend on the dimension of the
problem as much as the latter do. 

Under the assumption that $\nabla f$ is Lipschitz-continuous, i.e., there exists some $L>0$ such that
\begin{equation}\label{lipschitz}
\n{\nf{x} - \nf{y}}\l L \n{x-y},       
\end{equation}
one of the most simple methods for solving~\eqref{min} is the proximal gradient  method that
generates $(x_n)$ as
\begin{equation}
    \label{clas:prox}
    x_{n+1} = \prox_{\la_n g}(x_n  - \la_n \nf{x_n}),
\end{equation}
where $\la_n \in (0, \frac 2 L)$. 

There are several methods \cite{tseng00,cruz15} that do not require
condition~\eqref{lipschitz}. Our linesearch procedure in some sense is
similar to them but is cheaper since it does not use a proximal
mapping. The more extended discussion concerning this will be
presented in Section~\ref{sec:comparison}.  We underline that
problems, where \eqref{lipschitz} does not hold, take place, for
example, in barrier methods, entropy maximization, geometric
programming, image processing
\cite{cevher:14,cevher:14_2,nesterov07,borwein_zhu,boyd:convex_optim,pustelnik2010proximal}.

We also have to mention a very important class of two-step proximal
gradient methods that include inertial (heavy ball) methods introduced
by Polyak in~\cite{polyak}  and accelerated proximal methods,
pioneered by seminal work of Nesterov \cite{nesterov1983} and further
developed in~\cite{nesterov07,fista,tseng08} for a problem of
composite minimization.  This class enjoys an improved convergence
rate compared with classical proximal gradient
method~\eqref{clas:prox}. For all these methods
condition~\eqref{lipschitz} is also important.

In order to see why assumption \eqref{lipschitz} is so crucial for most optimization methods,
consider \eqref{clas:prox} in more detail.  There are two classical approaches of deriving the
proximal gradient method. The first one consists in the interpretation
of \eqref{clas:prox} with fixed $\la = \la_n\in (0,\frac 2 L)$ as 
forward-backward method. Condition~\eqref{lipschitz} is necessary to establish that operator
$\id - \la \nabla f$ is firmly nonexpansive (Baillon-Haddad theorem~\cite{baucomb}). After this we
can simply deduce that $\prox_{\la g}(\id - \la \nabla f)$ is averaged. Then a convergence of $(x_n)$
to a minimizer of $\Phi(x)$ follows from the celebrated Krasnosel'skii-Mann theorem~\cite{baucomb}.

In order to  derive \eqref{clas:prox} in a different way, we need  the following inequality
\begin{equation}\label{dl}
    |f(y)  - f(x) - \lr{\nf{x},y-x}|\l \frac{L}{2}\n{x-y}^2
\end{equation}
that is well-known as descent lemma \cite{baucomb}. Note that this
lemma holds for any smooth function $f$ that
satisfies~\eqref{lipschitz}.  We point out that analysis of
accelerated proximal methods is much more sophisticated and it can not
be interpreted as forward-backward splitting iteration. Nevertheless
one of the main ingredients in their analysis is inequality~\eqref{dl}
and hence the assumption~\eqref{lipschitz} is also necessary for them.

Among first-order methods there has been always some trade-off between
methods with fixed stepsize and ones with variable stepsizes. The
former are simpler and require less computation per iteration,
however, they require to know the Lipschitz constant $L$. Usually we
are able only to estimate this constant from the above (and even this
task sometimes can be very challenged), moreover this estimation is
often quite conservative, so the method will use tiny steps. Methods
with variable stepsizes in each iteration run some linesearch
procedure in order to find appropriate stepsize. They are more
flexible and usually allow to use larger steps than what is predicted
by the Lipschitz constant. At this moment there are a lot of possible
linesearch procedures and adaptivity techniques for \eqref{clas:prox}
or some particular cases of \eqref{clas:prox}, see
\cite{scheinberg14,bb-gradient,birgin-raydan,tseng00,nesterov07,fista,cruz15,donoghue-candes:2013,lin-xiao:2014}.
 
Our method seems to fill in this trade-off: it is very simple, it uses
variable and nonmonotone stepsizes, and linesearch procedure is quite
cheap and flexible. This significantly differs it from the most known
methods. For example, each inner iteration of the popular Armijo-like
linesearch procedure for \eqref{clas:prox} proposed in~\cite{fista}
requires evaluation of $f$ and $\prox_{\la_n g}$. Moreover, in order
to provide convergence of generated sequence $(x_n)$, the sequence of
stepsizes $(\la_n)$ must be nonincreasing.

On the other hand, the flexibility of stepsizes $(\la_n)$ in our
methods causes difficulties to get a nonergodic convergence rate of the proposed
algorithms. We hope the presence of numerical experiments in the
paper makes this lack up. Roughly speaking, the general picture of
applicability of our methods is the following. In cases when local
Lipschitz constant of $\nabla f$ changes drastically, that is $f$ has
very different curvature in different directions, then a global
Lipschitz constant can not be a good prediction and our methods will
benefit from using the local information of $\nabla f$. In turn, when
$\nabla f$ is rather flat, i.e.~local Lipschitz constant of $\nabla f$
does not change too much, our method will be in the worse case
comparing to other methods, since the latter allow to take stepsizes larger
or/and they may enjoy a better complexity rate.
Clearly, the first case is the most difficult in optimization, since
it includes problems with highly nonlinear $\nabla f$ or linear but
ill-conditioned $\nabla f$.

\subsection{Variational inequality perspective} This subsection
concerns with a more general case when $F$ is not a gradient of a
convex function. A general approach to solve \eqref{gvi} consist in solving of a
sequence of the simpler variational
inequalities~\cite{cohen,harker-pang}. We concentrate ourselves on the
most simple case of this approach: projected (proximal) methods.

When, for example, $F$ satisfies cocoercivity
assumption (that is stronger than just monotonicity) then most methods
from optimization framework can be still applied to this case. In
particular, this holds for the proximal gradient method
(forward-backward method) \cite{passty79,lions-mercier},
inertial method \cite{moudafi:03,pock:inertial}. However, those
methods do not converge when $F$ is monotone.

When $g(x) = \d_C(x)$, variational inequality \eqref{gvi} reduces to
\begin{equation}
    \label{vip}
    \lr{F(x^*), x - x^* }  \geq 0 \quad \forall x \in C,
\end{equation}
where $C\subseteq \E $ is a closed convex set.
For this specific case Korpelevich \cite{korpel:76} proposed the
extragradient method

\begin{gather}
\begin{aligned}\label{korpel}
    y_{n} & =P_C(x_n-\la F(x_n))\\
    x_{n+1} & =P_C(x_n-\la F(y_n)),
\end{aligned}
\end{gather}
where $\la \in (0,\frac 1 L)$. A bit different approach was proposed 
 by Popov \cite{popov:80}
\begin{gather}
\begin{aligned}\label{popov}
x_{n+1} &=P_C(x_n-\la F(y_n))\\
y_{n+1} & =P_C(x_{n+1}-\la F(y_n)),
\end{aligned}
\end{gather}
where $\la \in (0,\frac{1}{3L}]$.
Note that the latter method needs only one value of $F$ per iteration,
though it uses a smaller stepsize.

Both Korpelevich' and Popov's methods gave birth to a fruitful
research
\cite{mal-sem,malitsky_journ_glob_opt,lyashko11,khobotov,iusem:97,reich:2011,solodov1996modified,solodov:1999,semenov:extra_lipschitz,gibali15}
where there have been proposed different improvements: linesearch
procedures or/and avoiding of Lipschitz-continuity assumption,
decreasing a number of metric projections, etc. Actually, the basic schemes
\eqref{korpel} and \eqref{popov} can be applied to a general problem
\eqref{gvi}. However, this is not always the case for their
extensions.

In turn, problem~\eqref{gvi} can be formulated as a more general
problem of a monotone inclusion. In this case one may apply the
Tseng's forward-backward-forward method~\cite{tseng00}
\begin{gather}
\begin{aligned}\label{tseng2}
    y_n &  = \prox_{\la g}(x_n-\la F(x_n))\\
     x_{n+1} & = y_n +\la (F(x_n)-F(y_n)),
 \end{aligned}
\end{gather}
where $\la \in (0, \frac{1}{L})$. Note that the linesearch proposed in
the same paper~\cite{tseng00} allows us to require only continuity of
$F$. However, even with fixed steps the method uses two values
of $F$ per iteration. 

As in the previous subsection, the algorithms for \eqref{gvi} or \eqref{vip} that have practical
interest use some linesearch procedures to find $\la_n$ in each
iteration. The most popular choice is the Goldshtein-Armijo-type
stepsize rule
\cite{khobotov,tseng00,solodov:1999,solodov1996modified}, which requires  evaluation of $F$
and $\prox_{g}$ in each of inner iterations.  For method
\eqref{tseng2} we will consider such
implementation  in more details in Section~\ref{sec:comparison}.

Recently, in \cite{malitsky15} there was proposed the reflected projected gradient
method for problem~\eqref{vip}. When stepsize $\la$ is fixed, it generates a sequence $(x_n)$
by
$$x_{n+1} = P_C(x_n - \la F(2x_n - x_{n-1})),$$
where $\la \in (0, \frac{\sqrt 2 - 1}{L})$. This scheme is much
simpler than \eqref{korpel}, \eqref{popov}, or \eqref{tseng2} but the most important
that it gives a very efficient way to incorporate a linesearch
procedure. In~\cite{malitsky15} one of such ideas was applied and
numerical results approved its efficiency. However, the proposed
scheme was quite complicated and one of the goals of this paper is to
propose simpler schemes that, in addition, can be applied to a more 
general problem than \eqref{vip}.

\section{Main part} \label{sec:main}
The following assumptions are made throughout the paper:
\begin{itemize}
    \item[{\bf A1}] $F\colon \E\to \E$ is locally Lipschitz
    continuous and monotone.
    \item[{\bf A2}] $g\colon \E \to (-\infty, \infty]$ is proper l.s.c.  convex function.
    \item[{\bf A3}] $g |_{\dom g}$ is a continuous function.
    \item[{\bf A4}] The solution set of \eqref{gvi}, denoted by $X^*$, is nonempty.
\end{itemize}
The assumption A3 seems to be not quite usual, though it is very
general. Clearly, it fulfills for any $g$ with open $\dom g$ (this
includes finite-valued functions) or for an indicator $\d_C$ of some
closed convex set $C$. Moreover, when $\E = \R$ A2 implies A3
(Corollary 9.15, \cite{baucomb}). By this, every  separable function
that satisfies A2 also satisfies A3.

The following two lemmas are classical. For their proofs we refer to~\cite{baucomb}.
\begin{lemma}\label{prox_charact}
Let $g\colon \E\to (-\infty,+\infty]$ be a convex function, $x\in \E$. Then
$p = \prox_{g}(x)$ if and only if
$$\lr{p-x, y - p}\g g(p) - g(y)\quad \forall y\in \E.$$
\end{lemma}

\begin{lemma}\label{sol_char}
Let $\la >0$ and (A2) holds. Then $x^*$ is a solution of \eqref{gvi} if and only if
 \begin{equation*}
     x^* = \prox_{\la g}(x^* - \la F(x^*)).
 \end{equation*}
\end{lemma}
Next lemma is obvious.
\begin{lemma}\label{lim_seq}
  Let $(a_n)$, $(b_n)$ be two nonnegative real sequences such that
$$a_{n+1}\leq a_n - b_n.$$
Then $(a_n)$ is bounded and $\lim_{n\to \infty}b_n = 0$.
\end{lemma}

\subsection{Algorithm 1} \label{sec:alg1}
Firstly, for simplicity, we consider a particular case of \eqref{gvi}
when $g(x)=\d_C(x)$ for a closed convex set $C\subseteq \E$. Now the
problem becomes to find $x^*\in C$ such that
\begin{equation}
    \label{vip2}
    \lr{F(x^*), x - x^* }  \geq 0 \quad \forall x \in C.
\end{equation}







\begin{algorithm}[!ht]\caption{\textit{}}
    \label{alg:A1}
\begin{algorithmic}
   \STATE {\bfseries Initialization:} 
   Choose  $\a \in(0, \sqrt 2 -1)$, $\la_{max} > 0$, $\sigma \in (0,1)$,
 $x_1,x_0, y_0 \in \E$, $\la_0 > 0$. Set $\t_0 = 1$. 
  \STATE {\bfseries Main iteration:} 
   	\STATE 1. For given $x_n$, $x_{n-1}$, $y_{n-1}$, $\la_{n-1}$,
$\t_{n-1}$ set $i = 0$ and run
	\STATE ~~~ \textbf{Linesearch:}
    \STATE ~~~1.a. Take $\t_n = \sigma^i$ and $y_n = x_n + \t_n (x_n -
x_{n-1})$.
	\STATE ~~~1.b. Choose the largest $\la_n \leq \min\{\frac{1+\t_{n-1}}{\t_n}
\la_{n-1}, \la_{max}\} $ such that
\begin{equation}\label{lipsch_main}
\n{\la_n F(y_n)- \la_{n-1} \t_n F(y_{n-1})}\leq \a \n{y_n - y_{n-1}}.
\end{equation}
	\vspace{-3ex}
	\STATE ~~~1.c. Break linesearch if such $\la_n$ exists. Otherwise,
    set $i:= i+1$ and go to 1.a.
	\STATE ~~\textbf{End of linesearch}
    \STATE 2. Compute $x_{n+1} = P_C (x_n - \la_n F(y_n))$.
\STATE{\bfseries Output:}  
Return $x_n$ and $y_n$.
\end{algorithmic}
\end{algorithm}

We need  $\la_n \leq \la_{max}$ to
ensure that $(\la_n)$ is bounded.  Inequality \eqref{lipsch_main} gives
us something similar to an estimation that we usually get from
Lipschitz continuity of $F$. It is easy to see that finding
the largest $\la_n$ that satisfies \eqref{lipsch_main} is equivalent
to solving a quadratic equation, thus it can be found explicitly. 
Note that update of the inner loop requires only computation of
$F$. 

First, let us show that Algorithm~1 is well-defined.
\begin{lemma} \label{well-def1}
    The linesearch in Algorithm~1 always terminates.
\end{lemma}
\begin{proof}
    Suppose the assertion of the lemma is false.  Let $D = \conv\{x_n,
    2x_n - x_{n-1}, y_{n-1} \}$. Since $F$ is locally
    Lipschitz-continuous, it is Lipschitz-continuous on $D$ (because
    $D$ is a bounded set). Hence, there exists $L$ such that
    $$\n{F(y_n) - F(y_{n-1})}\leq L \a \n{y_n - y_{n-1}}.$$
    Note that $y_n \in D$ for any $\t_n \in (0,1]$. Then, in order to get a contradiction, it remains to take $\t_n < \frac{1}{\la_{n-1}L}$ and set
    $\la_n = \t_n \la_{n-1}$.
\end{proof}

\begin{lemma}
    \label{basic_lemma1}
    For $(x_n)$, $(y_n)$, generated by Algorithm 1, and $x\in C$ the
    following inequality holds
\begin{multline}\label{alg1:ineq-1}
 \n{x_{n+1}-x}^2 \l 
    \n{x_n-x}^2 + 2\a\n{y_n-y_{n-1}}\n{x_{n+1}-y_n}
    \\ - 
\n{y_n-x_n}^2 - \n{x_{n+1}-y_n}^2  - 2\la_n \lr{F(y_n), y_n -x}.
\end{multline}
\end{lemma}
\begin{proof}
        By Lemma~\ref{prox_charact},
    \begin{equation}
        \label{p1}
          \lr{x_{n+1} - x_{n} + \la_n F(y_{n}), x-x_{n+1}} \g 0\quad
          \forall x\in C.
    \end{equation}
Similarly, for the previous iterate we have
\begin{equation*}
    \lr{x_{n} - x_{n-1} + \la_{n-1} F(y_{n-1}), x - x_n} \g 0 \quad
    \forall x\in C. 
\end{equation*}
Taking in the above inequality $x = x_{n+1} \in C$ and then $x =
x_{n-1} \in C$, we obtain
\begin{align}
    \lr{x_{n} - x_{n-1} + \la_{n-1} F(y_{n-1}), x_{n+1} - x_n} & \g
    0, \label{prev1}\\
        \lr{x_{n} - x_{n-1} + \la_{n-1} F(y_{n-1}), x_{n - 1} - x_n}
        & \g 0. \label{prev2}
\end{align}
Multiplying~\eqref{prev2} by $\t_n$ and adding it to \eqref{prev1} gives us
\begin{equation*}
    \lr{x_{n} - x_{n-1} + \la_{n-1} F(y_{n-1}), x_{n+1}-y_{n}} \g 0 .
\end{equation*}
From $\t_n(x_n - x_{n-1}) = y_n - x_n$ we get
\begin{equation}
    \label{p2}
    \lr{y_n - x_{n}  + \t_n \la_{n-1} F(y_{n-1}), x_{n+1}-y_{n}} \g 0 .
\end{equation}
Summing \eqref{p1} and \eqref{p2}, we obtain
\begin{multline}
    \label{e4}
 \lr{x_{n+1} - x_n , x-x_{n+1}} + \lr{y_{n} - x_{n}, x_{n+1}-y_{n}} \\
 + \lr{\la_n F(y_n) - \t_n \la_{n-1}F(y_{n-1}), y_n -
 x_{n+1}} \g \la_n \lr{F(y_n), y_n-x} \quad \forall x\in C.   
\end{multline}
By the cosine rule, we derive
\begin{multline} 
    \n{x_n -x}^2 -  \n{x_{n+1}-x}^2 -  \n{y_n-x_n}^2
    - \n{x_{n+1}-y_n}^2 \\ +2 \lr{\la_n F(y_n) - \t_n \la_{n-1}F(y_{n-1}), y_n -
 x_{n+1}} \geq 2 \la_n \lr{F(y_n), y_n-x} \quad \forall x\in C.
\end{multline}
Taking into account \eqref{lipsch_main}, we get the desired inequality
\eqref{alg1:ineq-1}.
\end{proof}
 
\begin{lemma}    \label{separ}
Assume that $(x_n)$, generated by Algorithm 1, is bounded. Then  $\limsup_{n\to \infty} \la_{n}
> 0$.
\end{lemma}
\begin{proof}
Evidently, the sequence $(y_n)$ is bounded as well. Since $F$ is
Lipschitz-continuous on bounded sets, there exists $L>0$ such that
$$\n{F(y_n) - F(y_{n-1})}\leq \a L \n{y_n - y_{n-1}} \quad \forall
n\in \N.$$
From the construction of $(\la_n)$ it can be seen easily that if we
have $\la_{n-1} < \frac 1 L$ then $\t_n = \sigma^{0} = 1$ and $\la = \la_{n-1}$
satisfy inequality
$$\n{\la F(y_n) - \t_n\la_{n-1} F(y_{n-1})} \leq \a \n{y_{n} - y_{n-1}}.$$
In other words, the linesearch terminates after one iteration. Since we seek
the largest $\la \in (0, \frac{1+\t_{n-1}}{\t_n}\la_{n-1}]$,  we have $\la_n
\geq \la_{n-1}$. 

Now, on the contrary, assume that $\lim_{n \to \infty} \la_n =
0$. Hence, there exists $n_0$ such that $ \la_n < \frac{1}{L}$ for
all $n \geq n_0$.  Let $n > n_0 $. As $\la_{n-1} < \frac{1}{L}$, we get $\la_n\geq
\la_{n-1}$. But $\la_n< \frac{1}{L}$ as well, so again we have that
$\la_{n+1}\geq \la_{n}$. By induction we conclude that $(\la_n)_{n\geq
n_0}$ is nondecreasing and thus can not converge to zero. This
contradicts to our assumption.
\end{proof}

\subsection{Algorithm 2} \label{sec:alg2}
For a general problem \eqref{gvi} we propose the following algorithm.












\begin{algorithm}[!ht]\caption{\textit{}}
    \label{alg:A2}
\begin{algorithmic}
    \STATE {\bfseries Initialization:} Choose
    $\a \in(0, \sqrt 2 -1)$, $\la_{max} > 0$, $\sigma \in(0,1)$
    $x_1,x_0, y_0 \in \E$, $\la_0 > 0$. Set $\t_0 = 1$.
    \STATE {\bfseries Main iteration:} 
    \STATE 1. For given $x_n$, $x_{n-1}$, $y_{n-1}$, $\la_{n-1}$,
    $\t_{n-1}$ set $i = 0$ and run
    \STATE ~~~ \textbf{Linesearch:}
    \STATE ~~~1.a. Set
    \vspace{-3ex}
    \begin{align*}
    \t_{n} & =
    \begin{cases}
        \sqrt{1+\t_{n-1}} \sigma^i, \text{ if } \la_{n-1}\leq \frac 1 2 \la_{max}, \\
        \sigma^i, \text{otherwise},
    \end{cases} \\  
    y_n & = x_n + \t_n (x_n - x_{n-1}),\\
    \la_n & = \t_n \la_{n-1}.
    \end{align*}
    \vspace{-3ex}
    \STATE ~~~1.b. Break linesearch loop if
    \begin{equation}\label{lipsch_main2}
        \la_{n}\n{F(y_{n})-F(y_{n-1})}\leq \a \n{y_n - y_{n-1}}.
    \end{equation}
    \vspace{-3ex}
    \STATE \qquad Otherwise,  set $i:= i+1$ and go to 1.a.
    \STATE ~~~\textbf{End of linesearch}
    \STATE 2. Compute $x_{n+1} = \prox_{\la_n g}(x_n - \la_n F(y_n))$.
    \STATE{\bfseries Output:}  
    Return $x_n$ and $y_n$.
\end{algorithmic}
\end{algorithm}

So, basically, the linesearch procedure finds such
$\t_n\in(0,\sqrt{1+\t_{n-1}}\,]$ (trying to choose the larger one)
that $\la_n = \t_n \la_{n-1}$ satisfies the ``local Lipschitz''
condition~\eqref{lipsch_main2}.  On the one hand, we want to have
$\t_n\geq 1$, since this gives us possibility at least theoretically
to increase the stepsize from iteration to iteration. On the other
hand, we have to ensure that $(\la_n)$ will not be larger than
$\la_{max}$. These caused a bit complicated formula for $\t_n$.

Although \eqref{gvi} with $g(x)=\d_C(x)$ is precisely \eqref{vip},
Algorithm 1 in this case does not coincide with Algorithm
2. The former is more flexible since it does not apply such a
restriction on  stepsizes $\la_n$ as the latter does.

We want to point out that when $F$ is $L$-Lipschitz-continuous,
instead of running linesearch procedure, we can use
a fixed stepsize $\la \in (0, \frac{\a}{L})$ and take $\t_n = 1$ in each
iteration of Algorithm 2. By this we recover a basic algorithm in \cite{malitsky15}.
 
As before, let us show that Algorithm 2 is well-defined.
\begin{lemma}\label{well-def2}
The linesearch in Algorithm 2 always terminates.
\end{lemma}
\begin{proof}
The  proof is very similar to the proof of Lemma \ref{well-def1}.  The
main distinction is that now we have to set $D = \conv\{x_n,
(1+\varphi)x_n - \varphi x_{n-1}, y_{n-1} \}$, where $\varphi =
\frac{\sqrt 5 + 1}{2}$, and notice that $\t_n \leq
\varphi$ for all $n\in \N$.
\end{proof}

\begin{lemma}
    \label{basic_lemma2}
    For $(x_n)$, $(y_n)$ defined in Algorithm 2 and $x\in \E$ the following
    inequality holds
    \begin{multline}
    \label{alg2:ineq-1}
    \n{x_{n+1}-x}^2\leq \n{x_n -x}^2 - \n{y_n- x_n}^2
    -\n{x_{n+1}-y_n}^2 \\ +2 \a \n{y_n-y_{n-1}}\n{x_{n+1}-y_n} -
    2\la_n((1+\t_n)g(x_n) - \t_{n}g(x_{n-1})-g(x)) \\ - 2\la_n
    \lr{F(y_n), y_n - x}.
\end{multline}
\end{lemma}
The general idea of the following proof is very similar to the
previous one.
\begin{proof}
        By Lemma~\ref{prox_charact}
\begin{equation}
    \label{eq:2.1}
    \lr{x_{n+1} - x_n + \la_n F(y_n), x-x_{n+1}} \g \la_n (g(x_{n+1})-g(x)) \quad \forall x\in \E.
\end{equation}
Similarly,
$$    \lr{x_{n} - x_{n-1} + \la_{n-1} F(y_{n-1}), x-x_{n}} \g
\la_{n-1} (g(x_{n})-g(x)) \quad \forall x\in \E.$$
After substitution in the last inequality $x = x_{n+1}$ and  $x =
x_{n-1}$ we obtain
\begin{align*}
     \lr{x_{n} - x_{n-1} + \la_{n-1} F(y_{n-1}), x_{n+1}-x_{n}} & \g \la_{n-1}
(g(x_{n})-g(x_{n+1})),\\
\lr{x_{n} - x_{n-1} + \la_{n-1} F(y_{n-1}), x_{n-1}-x_{n}} & \g \la_{n-1}
(g(x_{n})-g(x_{n-1})).
\end{align*}
Multiplying the last inequality by $\t_n$ and then adding it to the previous ones yields
\begin{equation}
    \label{2.xn}
    \lr{x_{n} - x_{n-1} + \la_{n-1} F(y_{n-1}), x_{n+1}-y_{n}} \g \la_{n-1}
((1+\t_n)g(x_n) - g(x_{n+1}) - \t_n g(x_{n-1}) ).
\end{equation}
From $\t_n(x_n-x_{n-1}) = y_n - x_n$ and $\la_n = \t_n \la_{n-1}$ we
get
\begin{equation}
    \label{2.xn2}
    \lr{y_n-x_n + \la_{n} F(y_{n-1}), x_{n+1}-y_{n}} \g \la_{n}
((1+\t_n)g(x_n) - g(x_{n+1}) - \t_n g(x_{n-1}) ).
\end{equation}

Adding \eqref{eq:2.1} to \eqref{2.xn2} gives us
\begin{multline}
    \label{add_all}
 \lr{x_{n+1} - x_n , x-x_{n+1}} + \lr{y_{n} - x_{n}, x_{n+1}-y_{n}} 
 + \la_n \lr{F(y_n) - F(y_{n-1}), y_n -
 x_{n+1}} \\ \g \la_n ((1+\t_n)g(x_n) - \t_n g(x_{n-1}) - g(x)) + \la_n
 \lr{F(y_n), y_n-x}.   
\end{multline}
Using the cosine rule and \eqref{lipsch_main2}, we obtain
\begin{multline}
    \label{eq:11}
    \n{x_{n+1}-x}^2\leq \n{x_n -x}^2 - \n{y_n- x_n}^2
    -\n{x_{n+1}-y_n}^2 \\ +2 \a \n{y_n-y_{n-1}}\n{x_{n+1}-y_n} -
    2\la_n((1+\t_n)g(x_n) - \t_{n}g(x_{n-1})-g(x)) \\ - 2\la_n
    \lr{F(y_n), y_n - x},
\end{multline}
that finishes the proof.
\end{proof}

For Algorithm~2 we can prove a stronger result than Lemma~\ref{separ}.
\begin{lemma}
    \label{separ2} 
Assume that the sequence $(x_n)$, generated by Algorithm 2, is bounded. Then  $\liminf_{n\to \infty} \la_{n}
> 0$.
\end{lemma}
\begin{proof}
Since $(x_n)$ is bounded, there exists $L>0$ such that
$$\n{F(y_n) - F(y_{n-1})}\leq \a L \n{y_n - y_{n-1}}.$$
Without loss of generality assume that $\la_0 > \frac{\sigma}{L}$. We
show that from $\la_{n-1} > \frac{\sigma}{L}$ follows
$\la_n>\frac{\sigma}{L}$. Clearly, if $\la_n  \leq
\frac 1 L$ then \eqref{lipsch_main2} holds.
Suppose that $\la_n = \sqrt{1+\t_{n-1}}\sigma^i \la_{n-1}$ for some
$i\in \Z^+$. If $i = 0$ then it is obvious that
$\la_n>\la_{n-1}>\frac{\sigma}{L}$. If $i>0$ then by the construction
of the linesearch  $\la_n' =
\sqrt{1+\t_{n-1}}\sigma^{i-1} \la_{n-1}$ does not satisfied  \eqref{lipsch_main2}. This means that $\la_n'> \frac
1 L$ and hence, $\la_n > \frac{\sigma}{L}$. 
\end{proof}

\subsection{Proof of convergence} \label{sec:proof}
For generality we will write
$$\Psi(x, y):=\lr{F(x), y-x}+g(y)-g(x),$$
where in case of Algorithm 1 we suppose that $g(x)=\d_C(x)$.
It is clear that both problems~\eqref{vip} and~\eqref{gvi} are
equivalent to finding $\x\in  
\E$ such that $\Psi(\x, x)\geq 0$ for all $x\in \E$.

\begin{lemma}\label{lemma:main_ineq}
    Let $(x_n)$, $(y_n)$ be generated by either Algorithm 1
    or 2 and let $\x \in X^*$. Then the following inequality holds
    \begin{align}\label{main_ineq}
     \n{x_{n+1}-\x}^2 \leq {} & \n{x_n-\x}^2 - (1-\a (1+\sqrt 2))\n{y_n-x_n}^2 \nonumber \\ & -
    (1-\sqrt 2\a) \n{x_{n+1}-y_n}^2 + \a \n{x_n-y_{n-1}}^2\nonumber \\ & - 2\la_n
    (1+\t_n)\Psi(\x, x_n) +  2\la_{n-1}(1 + \t_{n-1}) \Psi(\x, x_{n-1}).
\end{align}
\end{lemma}
\begin{proof}
    Monotonicity of $F$ yields
    \begin{align}
    \label{mono}
      \la_n \lr{F(y_n), y_n - \x} & \geq \la_n \lr{F(\x), y_n - \x}
                                   \nonumber \\
      & = \la_n ((1+\t_n)\lr{F(\x), x_n-\x} - \t_n \lr{F(\x), x_{n-1}-\x}).
    \end{align}
Taking $x=\x$ and using the above, we can rewrite  both
\eqref{alg1:ineq-1} and \eqref{alg2:ineq-1} as one inequality
\begin{multline} 
    \n{x_{n+1}-\x}^2\leq \n{x_n -\x}^2 - \n{y_n-x_n}^2
    - \n{x_{n+1}-y_n}^2 \\ +2 \a \n{y_n-y_{n-1}}\n{x_{n+1}-y_n} -
    2\la_n((1+\t_n)\Psi(\x, x_n) - \t_n \Psi(\x, x_{n-1})).
\end{multline}
Note that in both cases we have that $\la_n \t_n \leq
(1+\t_n)\la_{n-1}$. Since $\Psi(\x, x_{n-1})\geq 0$, it follows
\begin{multline}\label{ineq-22}
    \n{x_{n+1}-\x}^2\leq \n{x_n -\x}^2 - \n{y_n-x_n}^2
    - \n{x_{n+1}-y_n}^2 \\ +2 \a \n{y_n-y_{n-1}}\n{x_{n+1}-y_n} -
    2\la_n(1+\t_n)\Psi(\x, x_n) + 2\la_{n-1}(1+\t_{n-1}) \Psi(\x,
    x_{n-1}). 
\end{multline}
It only remains to estimate $2\a \n{y_{n}-y_{n-1}}\n{x_{n+1}-y_n}$. For
this we use the estimation from \cite{malitsky15}.
\begin{multline}\label{estim}
2\a \n{y_n-y_{n-1}}\n{x_{n+1}-y_n}   \leq
\a \bigl(\frac{1}{\sqrt 2}\n{y_n-y_{n-1}}^2 + \sqrt{2} \n{x_{n+1}-y_n}^2\bigr) \\  \leq
\frac{\a}{\sqrt 2}\bigl(\n{y_n - x_n} + \n{x_n - y_{n-1}}\bigr)^2 + \sqrt 2 \a \n{x_{n+1}-y_n}^2\\
\leq  \frac{\a}{\sqrt 2}\bigl((2+\sqrt 2)\n{y_n-x_n}^2+\sqrt 2\n{x_n-y_{n-1}}^2\bigr) +  \sqrt{2} \a \n{x_{n+1}-y_n}^2 \\ 
=\a (1+\sqrt 2)\n{y_n - x_n}^2+ \a \n{x_n-y_{n-1}}^2 +  \sqrt{2} \a \n{x_{n+1}-y_n}^2.
\end{multline}
Combining \eqref{ineq-22} and \eqref{estim}, we get the desirable inequality \eqref{main_ineq}.
\end{proof}

\begin{theorem}\label{main_th}
    Let sequences $(x_n)$ and $(y_n)$ be generated by either Algorithm
    1 or 2. Then $(x_n)$ and $(y_n)$ converge to a solution of
    \eqref{gvi}.
\end{theorem}
\begin{proof}
     Let us show that the sequence $(x_n)$ is bounded. Fix any $\x \in
     X^*$.      For $n\geq 1$  set
\begin{align*}
a_{n+1} & = \n{x_{n+1}-\x}^2  + \a \n{x_{n+1}-y_n}^2 + 2\la_n(1 +
\t_n)\Psi(\x, x_n)\\
b_n & = (1-\a (1+\sqrt 2))\bigl(\n{x_n-y_n}^2 + \n{x_{n+1}-y_n}^2\bigr).
\end{align*}
It is easy to see that \eqref{main_ineq} is equivalent (in a new
notation) to
\begin{equation}\label{ab}
    a_{n+1}\l a_n - b_n.
\end{equation}
Evidently, $a_n\g 0$ and $b_n\g 0$. Hence, by Lemma~\ref{lim_seq} we conclude that $(a_n)$ is bounded
and $\lim_{n\to \infty}b_n = 0$. This means that $(\n{x_n-x}^2)$ is bounded as well as $(x_n)$
and $$\lim_{n\to \infty}\n{x_n - y_n} = 0, \quad \lim_{n\to \infty}\n{x_{n+1}- y_n} = 0. $$
From the above it also follows that $\lim_{n\to \infty}\n{x_{n+1} - x_n} = 0$ and $(y_n)$ is
bounded.

By Lemma~\ref{separ} or \ref{separ2} and by boundedness of $(x_n)$ there exists an increasing
sequence $(n_k)$ of positive numbers such that $(\la_{n_k})$ is separated from
zero and $(x_{n_k})$ converges to some $x^* \in \E$ as $k\to \infty$.  It
is clear that $(y_{n_k})$ also converges to that $x^*$. We show $x^* \in
X^*$.

From Lemma~\ref{prox_charact} it follows that
 $$ \lr{x_{n_k+1} - x_{n_k} + \la_{n_k} F(y_{n_k}), x-x_{n_k+1}} \g
\la_{n_k}( g(x_{n_k+1}) - g(x))
\quad \forall x\in \E$$
 or equivalently
 \begin{equation}\label{eq:12}
 \lr{\frac{x_{n_k+1} - x_{n_k}}{\la_{n_k}}, x-x_{n_k+1}} +
 \lr{F(y_{n_k}), x-x_{n_k+1}} \g g(x_{n_k+1}) - g(x) \quad \forall x\in \E.
\end{equation}
 Taking the lower limit in \eqref{eq:12} as $k \to \infty $ and using that $(\la_{n_k})$ is separated from
 zero, 
  $\n{x_{n_k+1}-x_{n_k}}\to 0$, and $g(x)$ is l.s.c., we obtain
 \begin{equation}\label{eq:13}
 \lr{F(x^*), x-x^*} \geq \liminf_{k\to \infty}g(x_{n_{k+1}}) - g(x)
 \geq g(x^*)-g(x) \quad \forall x\in \E.
\end{equation}
Hence, $x^* \in X^*$.

From \eqref{ab} we have that for any $\x\in X^*$ the sequence $(a_n)$ is monotone, hence, it is
convergent. Thus, taking $\x = x^*$ defined above, we get that the sequence
$$ a_{n+1}^* = \n{x_{n+1}-x^*}^2  +  \a \n{x_{n+1}-y_{n}}^2 +  2\la_{n}(1 +
\t_{n})\Psi(x^*, x_{n})$$
is convergent. As $(\la_n)$ is bounded and $\Psi(x^*,\cdot)$ is
continuous due to A3, $\lim_{n\to \infty} a_n^* = \lim_{k\to \infty
}a_{n_k +1}^* =
0$. Therefore, $\lim_{n\to \infty}\n{x_n - x^*} = 0$ and the proof is complete.
\end{proof}

As one can see, the last arguments were the only place where we used
A3. Without this assumption we are only able to show that all limits
points of $(x_n)$ belong to $X^*$.
 
\begin{remark}\label{rem:init}
Both Algorithm 1 and 2 require $\la_0 >0$ as input data. Although the
algorithms do not have any restriction on the initialization
procedure, we suggest to define $\la_0$ as follows.  Choose any
$x_1$ in a small neighborhood of the starting point $x_0$ and take the
largest $\la_0$ that satisfies
$$\la_0 \n{F(x_1)-F(x_0)} \leq \a \n{x_1 - x_0}.$$
\end{remark}

\subsection{Affine cases}
In this section we introduce some additional suggestions that can
simplify the proposed algorithms.
\begin{remark}\label{rem:F_affine}
 If $F$ is affine then instead of computing $F(y_n)$
in each iteration of linesearch procedures 1 or 2, we only need to
remember $F(x_n)$, $F(x_{n-1})$ and use that $F(y_n) =
(1+\t_n)F(x_n)-\t_n F(x_{n-1})$.
\end{remark}
Clearly, with this remark computational
complexity of Algorithm 1 or Algorithm 2  per iteration is almost the same as, for
example, projected gradient method (or proximal gradient method) with a fixed stepsize. Our algorithms
require a bit more simple vector-vector operations and a bit more memory.

\begin{remark}\label{rem:C_affine}
When  $C$ in \eqref{vip} is
an affine set, Algorithm 1 becomes simpler. Namely, we do not need the bounds
$\la_n \leq \frac{1 + \t_{n-1}}{\t_n} \la_{n-1}$ neither $\la_n \leq
\la_{max}$.
\end{remark}
In fact, the former bound
was required in our proof of Theorem~\ref{main_th}  to ensure that
$$\la_{n}\t_n \Psi(\x, x_{n}) \leq (1+\t_{n-1})\la_{n-1}
\Psi(\x, x_{n-1})$$ and the latter was used to show that
$\la_{n_k}\Psi(x^*, x_{n_k}) \to 0$.  However, when  $C$ is
affine, $F(\x) = 0$ and thus, $\Psi(\x, x) = 0$ for all $x\in C$. Therefore, both items above hold for any
choice of $\la_n$.

If we consider \eqref{vip} with affine map $F$ and affine set $C$ then
it is clear that Algorithm~1 will benefit all the advantages of the two remarks above.


\subsection{Rate of convergence} \label{sec:rate}
In this section we investigate the ergodic rate of convergence for the
sequence $(y_n)$ for Algorithm~1 and Algorithm~2. It is well-known that
such rate holds for the extragradient method
\cite{nemirovski2004prox,tseng08}. In these papers the authors propose
much more general methods among which the extragradient method was
only a particular example. However, those methods are more
complicated, they used fixed steps and they require Lipschitz
continuity of $F$.

We need the following error function (known as the dual gap function
\cite{pang:03,tseng08}):
\begin{equation}
    \label{eq:gap}
    e(y) = \max_{x\in \dom g} \Psi(x,y):=\lr{F(x),y-x}+g(y)-g(x), \quad y \in \E.
\end{equation}
The relation between this error function and problem \eqref{gvi} is given by the
following lemma. 
\begin{lemma}[see \cite{pang:03,tseng08}]
    \label{lemma:gap}
    $x^* \in X^*$ if and only if $x^* \in \dom g$ and $e(x^*)=0$.
\end{lemma}
Next theorem shows that we can use the above criteria  to find
$x^*$ with a desired accuracy.
\begin{theorem}\label{th:rate}
    Let $(x_n)$ and $(y_n)$ be the sequences generated by either Algorithm~1 or
    2. Define $\overline \la_N$ and $\overline{x}_N$ as
    $$\overline{\la}_N = \sum_{n=1}^N \la_n + \t_1\la_1, \quad \overline{x}_N =
    \frac{1}{\overline\la_N}\bigl(\sum_{n=2}^N \la_n y_n + (1+\t_1) \la_1 x_1\bigr).$$
    Then $\overline{x}_N\in \dom g$ and 
    \begin{equation}\label{rate_ergodic}
\Psi(x, \overline{x}_N) \leq
\frac{\n{x_1 - x}^2 + \a \n{x_1 - y_0}^2 + 2\la_1 \t_1
\Psi(x, x_0)}{\overline{\la}_N}.
\end{equation}
\end{theorem}
\begin{proof}
If in Lemma~\ref{lemma:main_ineq} we did not use
inequality~\eqref{ineq-22} we would get the following
\begin{align}\label{eq:rate1}
     \n{x_{n+1}-x}^2 \leq {} & \n{x_n-x}^2 - (1-\a (1+\sqrt 2))\n{y_n-x_n}^2 \nonumber \\ & -
    (1-\sqrt 2\a) \n{x_{n+1}-y_n}^2 + \a \n{x_n-y_{n-1}}^2\nonumber \\ & - 2\la_n
    (1+\t_n)\Psi(x, x_n) +  2\la_{n}\t_{n} \Psi(x, x_{n-1}),
\end{align}
from which follows
    \begin{multline}
    \label{eq:rate2}
    \la_n
    (1+\t_n)\Psi(x, x_n) -  \la_{n}\t_{n} \Psi(x, x_{n-1})  \leq
    \frac 1 2 \n{x_n -x}^2 - \frac 1 2\n{x_{n+1}-x}^2 \\
    -\frac{1-\sqrt 2\a}{2} \n{x_{n+1}-y_n}^2 + \frac{\a}{2} \n{x_n-y_{n-1}}^2.
\end{multline}
Summing \eqref{eq:rate2} over $n=1,\dots,N$, we obtain
\begin{align*}
   \la_N (1+\t_N)\Psi(x, x_N) & +
   \sum_{k=1}^{N-1}[\la_{k}(1+\t_{k})-\la_{k+1}\t_{k+1}]\Psi(x, x_{k})
    \\ &  \leq \frac 1 2\n{x_1 - x}^2 + \frac{\a}{2}
   \n{x_1 - y_0}^2 + \la_1 \t_1\Psi(x, x_0).
\end{align*}
Note that function $\Psi(x,\cdot)$ is convex and all the coefficients
in square brackets are nonnegative due to the assumption of
algorithms. Applying Jensen's inequality to the left side of the above
inequality, we get
\begin{equation*}
(\sum_{k=1}^N \la_k + \t_1\la_1)\Psi(x, \overline{x}_N) \leq \frac 1 2
\n{x_1 - x}^2 + \frac{\a}{2} \n{x_1 - y_0}^2 + \la_1 \t_1 \Psi(x, x_0),
\end{equation*} 
where $$\overline{\la}_N \overline{x}_N = \la_N(1+\t_N)x_N +
\sum_{k=1}^{N-1}[\la_k(1+\t_k)-\la_{k+1}\t_{k+1}]x_k =
\sum_{k=2}^N\la_k y_k + \la_1(1+\t_1)x_1.$$
Evidently, $\overline{x}_N \in \dom g$ which finishes the proof.
\end{proof}

Notice that $\overline{\la}_N \to \infty$ due to Lemma~\ref{separ}
and~\ref{separ2}.

When \eqref{gvi} is a particular case of a composite minimization
problem or  a saddle point problem,  inequality \eqref{rate_ergodic}
can be improved. For simplicity, we
show how to do this only
for the case of constrained optimization.
 
    If $F$ is a gradient of a convex differentiable function $f$,
    i.e. \eqref{gvi} is the result of $\min_{x\in C} f(x)$,
    then
    \begin{equation}\label{rate:conv}
        \lr{F(y_n), y_n -x} \geq f(y_n)-f(x).
    \end{equation}
    Instead of using \eqref{eq:rate1}, we consider Lemma~\ref{basic_lemma1} and Lemma~\ref{basic_lemma2}
    for $g(x)=\d_C(x)$ that give us identical inequality
    \begin{multline}  
    \label{eq:rate2.1}
    \n{x_{n+1}-x}^2\leq \n{x_n -x}^2 - \n{y_n- x_n}^2
    -\n{x_{n+1}-y_n}^2 \\ +2 \a \n{y_n-y_{n-1}}\n{x_{n+1}-y_n} -
    2\la_n \lr{F(y_n), y_n - x}.
\end{multline}
Applying \eqref{rate:conv} and estimation \eqref{estim}, we obtain
    \begin{multline}
    \label{eq:rate2.2}
\la_n(f(y_n) - f(x))\leq
     \frac 1 2\n{x_n -x}^2 - \frac 1 2\n{x_{n+1}-x}^2 - \frac{1-\sqrt
     2\a}{2} \n{x_{n+1}-y_n}^2 + \frac{\a}{2}
     \n{x_n-y_{n-1}}^2 
\end{multline}
Using the same arguments as in Theorem~\ref{th:rate}, we get
$$    f(\overline{x}_N) - f(x)\leq \frac{\n{x_1 - x}^2 + \a
   \n{x_1 - y_0}^2 + 2\t_1 \la_1 (f(x_0)-f(x))}{2 \overline\la_N} \quad \forall x\in C.$$

\section{Composite minimization} \label{sec:alg3}
When $F$ is a gradient of a convex
function, problem \eqref{gvi} is equivalent to a problem of a
composite minimization
\begin{equation}
    \label{comp_min} 
    \min_x \Phi(x):= f(x) + g(x),
\end{equation}
where we assume that
\begin{itemize}
    \item[{\bf A5}] $f\colon \E \to \R$ is a convex differentiable
    function with locally Lipschitz gradient $\nabla f$.
\end{itemize}
To highlight the specificity, instead of $F$ we will write $\nabla f$. We
denote $\Phi^*=\min_x \Phi(x)$. Throughout this section we suppose
that A2-A5 hold.









\begin{algorithm}[!ht]\caption{\textit{}}
    \label{alg:A3}
\begin{algorithmic} 
    \STATE {\bfseries Initialization:} Choose $\a \in(0, \sqrt 2 -1)$,
    $\la_{max} > 0$, $\sigma \in (0,1)$, $\th \in [1, 2]$,
    $x_1,x_0, y_0 \in \E$, $\la_0 > 0$. Set $\t_0 = 1$.  \STATE
    {\bfseries Main iteration:} \STATE 1. For given $x_n$, $x_{n-1}$,
    $y_{n-1}$, $\la_{n-1}$, $\t_{n-1}$ set $i = 0$ and run \STATE ~~~
    \textbf{Linesearch:} \STATE ~~~1.a. Set \vspace{-3ex}
\begin{align*}
    \t_{n} & = \begin{cases} \sqrt{\dfrac{1 + \th \t_{n-1}}{2\th -1}}
        \sigma^i, \text{ if } \la_{n-1}\leq \frac 1 2 \la_{max}, \\
        \sigma^i, \text{otherwise},
    \end{cases} \\
    y_n & = x_n + \t_n (x_n - x_{n-1}),\\
    \la_n & = (2-\frac 1 \th)\t_n \la_{n-1}.
\end{align*}
\vspace{-3ex}
	\STATE ~~~1.b. Break linesearch loop if
    \begin{equation}\label{lipsch_main3} 
\la_n \n{\nf{y_{n}}-\nf{y_{n-1}}}\leq \a (2-\frac 1 \th) \n{y_n - y_{n-1}}.
\end{equation}
\vspace{-3ex}
\STATE \qquad Otherwise,  set $i:= i+1$ and go to 1.a.

	\STATE ~~\textbf{End of linesearch}
    \STATE 2. Compute $x_{n+1} = \prox_{\la_n g}(x_n - \la_n \nf{y_n})$.
\STATE{\bfseries Output:}  
Return $x_n$ and $y_n$.
\end{algorithmic}
\end{algorithm}
Note that the stopping criteria of the linesearch procedure is the same as
in Algorithm~2:
\begin{equation} \label{eq:equiv_lip}
    \t_n \la_{n-1}\n{\nf{y_n} - \nf{y_{n-1}}} \leq \a \n{y_n -
y_{n-1}}.
\end{equation}
Moreover, for $\th = 1$ Algorithm~3 is identical to Algorithm~2. In
turn, for $\th > 1$ the stepsize $\la_n = (2-\frac 1
\th)\t_n\la_{n-1}$ is larger than in Algorithm~2.

Result stated in Lemma~\ref{well-def2}  hold for
Algorithm~3 as well. Since its proof is identical, we omit it. However, the main ingredient to prove a convergence of
$(x_n)$ differs from Lemmas~\ref{basic_lemma2} and~\ref{lemma:main_ineq}.
\begin{lemma}
    \label{basic_lemma3}
    For $(x_n)$, $(y_n)$ defined in Algorithm 3 and $x\in \E$ the following
    inequality holds
    \begin{align}\label{main_ineq3}
 \n{x_{n+1} - x}^2 \leq {} &  \n{x_n- x}^2 - (2\th - 1)\Bigl((1-\a (1+\sqrt 2))\n{x_n-y_n}^2 \nonumber \\ & -
 (1-\sqrt 2\a) \n{x_{n+1}-y_n}^2 + \a \n{x_n-y_{n-1}}^2\Bigr)\nonumber
 \\ & -2\la_n ((1+\th \t_n)\Phi(x_n) - \th \t_{n} \Phi(x_{n-1}) - \Phi(x)). 
\end{align}
\end{lemma}

\begin{proof}
    With the same arguments as in \eqref{eq:2.1} and \eqref{2.xn} we
    get
    \begin{equation}
        \label{3:1}
           \lr{x_{n+1} - x_n + \la_n \nf{y_n}, x-x_{n+1}} \g \la_n (g(x_{n+1})-g(x)) \quad \forall x\in \E,
    \end{equation}
    and
    \begin{multline*}
    \lr{x_{n} - x_{n-1} + \la_{n-1} \nf{y_{n-1}}, x_{n+1}-y_{n}} \g \la_{n-1}
((1+\t_n)g(x_n) - g(x_{n+1}) - \t_n g(x_{n-1}) ).
\end{multline*}
Using that  $\t_n(x_n-x_{n-1}) = y_n - x_n$ and $\la_n = (2-\frac 1 \th)\t_n \la_{n-1}$, we
get
\begin{multline}
    \label{3:3}
    \lr{(2-\frac 1 \th)(y_n-x_n) + \la_{n} \nf{y_{n-1}}, x_{n+1}-y_{n}} \g \la_{n}
((1+\t_n)g(x_n) - g(x_{n+1}) - \t_n g(x_{n-1}) ).
\end{multline}
By convexity of $f$,
\begin{equation} \label{conv}
 \lr{\nf{y_n}, y_n - x} \geq f(y_n)-f(x) \geq 
(1+\t_n)f(x_n)-\t_n f(x_{n-1}) - f(x).
\end{equation}
Summing \eqref{3:1}, \eqref{3:3}, and \eqref{conv}, multiplied by
$\la_n$, we obtain
\begin{multline}
    \label{add_all2}
 \lr{x_{n+1} - x_n , x-x_{n+1}} + (2-\frac 1 \th) \lr{y_{n} - x_{n}, x_{n+1}-y_{n}} \\
 + \la_n \lr{\nf{y_n} - \nf{y_{n-1}}, y_n -
 x_{n+1}} \g \la_n ((1+\t_n)\Phi(x_n) - \t_n \Phi(x_{n-1}) - \Phi(x)).   
\end{multline} 
Notice that for $\th = 1$ \eqref{add_all2} is very similar to 
\eqref{add_all}. Their distinction caused only by using convexity of
$f$ in \eqref{conv}. As usually, by the cosine rule we can rewrite the
above as
\begin{align}\label{3:4}
 \n{x_{n+1}-x}^2 \leq \n{x_n -x}^2 & - \n{x_{n+1}-x_n}^2 \nonumber\\
& + (2-\frac{1}{\th})\bigl(\n{x_{n+1}-x_n}^2 - \n{y_n-x_n}^2 -
 \n{x_{n+1}-y_n}^2 \bigr) \nonumber\\ & +2 \la_n\lr{\nf{y_n} - \nf{y_{n-1}}, y_n -
 x_{n+1}} \nonumber\\ & - 2\la_n ((1+\t_n)\Phi(x_n) - \t_n\Phi(x_{n-1}) - \Phi(x)).
\end{align}
 Let
\begin{multline}
    A = - (2-\frac 1 \th)\bigl(\n{y_n-x_n}^2 +
 \n{x_{n+1}-y_n}^2\bigr)  \\ + 2 \la_n\lr{\nf{y_n} - \nf{y_{n-1}}, y_n -
 x_{n+1}} - 2\la_n\t_n(\Phi(x_n) - \Phi(x_{n-1}) ).
\end{multline}
Then \eqref{3:4} is equivalent to
\begin{equation}\label{withA}
\n{x_{n+1}-x}^2 \leq  \n{x_n -x}^2 + (1 - \frac{1}{\th})
\n{x_{n+1}-x_n}^2 + A  - 2\la_n (\Phi(x_n)  - \Phi(x)).
\end{equation}
Recall that inequality \eqref{withA} holds for every $x\in \E$. Thus,
taking $x=x_n$, we obtain
\begin{equation*}
\n{x_{n+1}-x_n}^2\leq  (1 - \frac{1}{\th})
\n{x_{n+1}-x_n}^2 + A.
\end{equation*}
Hence, $\n{x_{n+1}-x_n}^2\leq \th A.$
Applying to~\eqref{withA}, this yields
\begin{multline*}
\n{x_{n+1}-x}^2\leq  \n{x_n -x}^2 + (1 - \frac{1}{\th})\th A
 + A  - 2\la_n (\Phi(x_n)  - \Phi(x)) \\ = \n{x_n -x}^2 + \th A
 - 2\la_n (\Phi(x_n)  - \Phi(x)).
\end{multline*} 
Using that $\la_n \n{\nf{y_n} - \nf{y_{n-1}}}\leq (2-\frac 1 \th) \a \n{y_n -
y_{n-1}}$, we deduce
\begin{align*}
  \n{x_{n+1}-x}^2\leq  \n{x_n -x}^2 & -  (2\th - 1)  \bigl(\n{y_n-x_n}^2 +
  \n{x_{n+1}-y_n}^2 \nonumber \\ & -  2\a\n{y_n - y_{n-1}}\n{x_{n+1}-y_n} \bigr) \nonumber \\
  & - 2\la_n ((1+\th \t_n)\Phi(x_n) - \th \t_n\Phi(x_{n-1}) - \Phi(x)) 
\end{align*} 
To complete the proof it only remains to use~\eqref{estim}.
\end{proof}

Unfortunately, we are not able to show that the whole sequence
$(\la_n)$ is separated from zero. This is because the first iteration
of the linesearch may start from $\t_n<1$. To show that $(\la_n)$ does
not converge to $0$, we need to apply a bit more complex arguments
than ones in Lemma~\ref{separ}.
\begin{lemma}
    \label{separ3}
    Assume that the sequence $(x_n)$, generated by Algorithm~3, is
    bounded. Then $\limsup_{n\to \infty} \la_{n} > 0$.
\end{lemma}
\begin{proof}
Since $(x_n)$ is bounded, there exists $L>0$ such that
$$\n{F(y_n) - F(y_{n-1})}\leq \a L \n{y_n - y_{n-1}}.$$
Also it is not difficult to show by induction that $\t_n < 2$ for
all $n$. Let $\la_{n-1}< \frac{1}{2L}$. We show that at least one of
$\la_n$ or $\la_{n+1}$ is larger or equal than $\la_{n-1}$. Evidently, from this the assertion of lemma
follows. 

On the contrary, assume that $\la_{n+j}<\la_{n-1}$ for $j=0,1$.
Due to $\la_{n+j}<\frac{1}{2L}$, $\t_{n+j}<2$, and \eqref{eq:equiv_lip}, the linesearch procedure in Algorithm 2
must terminate after the first iteration. This means that
$\t_n = \sqrt{\frac{1+\t_{n-1}\th}{2\th-1}}$ and $\t_{n+1}=
\sqrt{\frac{1+\t_{n}\th}{2\th-1}}$. From our assumption we have  $$\la_{n+1} = \left(2-\frac 1
\th\right)\t_{n+1}\la_{n} = \left(\frac{2\th
    -1}{\th}\right)^2\t_{n+1}\t_n\la_{n-1} < \la_{n-1}.$$
Using that $\t_n\geq \frac{1}{\sqrt{2\th-1}}$, we get
$$\frac{2\th -1}{\th^2}\sqrt{1+\t_n\th}< 1.$$
Note that $\t_n\th \geq \frac{\th}{\sqrt{2\th - 1}}\geq 1$. This implies $(2\th-1)\sqrt{2}<\th^2$. But
the latter inequality does not hold for any $\th\in [1,2]$.
This contradiction finishes the proof.
\end{proof}
 In fact, the upper bound for $\th$ can be enlarged, but
then the proof of Lemma~\ref{separ3} will be more
complicated. Perhaps larger $\th$ seems to be a better choice
because $(2-\frac 1 \th)$ will increase. However, in this case the
bound $\frac{1+\tau_{n-1} \th}{2\th-1}$ will decrease and in the
result we may get even smaller
$\la_n$. So, one can see $\th=2$ as a trade-off between those two
bounds. Numerical experiments also approved $\th =2$ as the best choice.

\begin{theorem}\label{main_th2}
    Let sequences $(x_n)$ and $(y_n)$ be generated by
    Algorithm~3. Then $(x_n)$ and $(y_n)$ converge to a
    solution of~\eqref{comp_min}.
\end{theorem}
\begin{proof}
From $\t_n \leq   \sqrt{\dfrac{1 + \th \t_{n-1}}{2\th -1}}$ and $\la_n =
(2-\frac 1 \th) \t_n \la_{n-1}$ it follows
\begin{equation}
    \label{3:5} 
    \la_n \th \t_n (\Phi(x_{n-1}) - \Phi^* ) \leq \la_{n-1}(1+\th \t_{n-1})(\Phi(x_{n-1}) - \Phi^*)
\end{equation} 
Applying \eqref{3:5} to \eqref{main_ineq3} with $x=\x\in X^*$, we get
\begin{align}\label{main_ineq4} 
 \n{x_{n+1} - \x}^2 \leq {} &  \n{x_n- \x}^2 - (2\th - 1)\Bigl((1-\a (1+\sqrt 2))\n{x_n-y_n}^2 \nonumber \\ & -
 (1-\sqrt 2\a) \n{x_{n+1}-y_n}^2 + \a \n{x_n-y_{n-1}}^2\Bigr)\nonumber
 \\ & -2\la_n (1+\th \t_n)(\Phi(x_n)- \Phi^*) +  2\la_{n-1}(1+\th \t_{n-1}) (\Phi(x_{n-1}) - \Phi^*) 
\end{align}
With sequences $(a_n)$ and $(b_n)$ given by
\begin{align*}
a_{n+1} & = \n{x_{n+1}-\x}^2  + (2\th -1)\a \n{x_{n+1}-y_n}^2 + 2\la_n(1 +
\th \t_n)(\Phi(x_n) - \Phi^*)\\
b_n & = (2\th-1)(1-\a (1+\sqrt 2))\bigl(\n{x_n-y_n}^2 + \n{x_{n+1}-y_n}^2\bigr)
\end{align*}
 the rest of the proof almost coincides with the proof of Theorem~\ref{main_th}.
\end{proof}

When $\nabla f$ is $L$--Lipschitz-continuous then Algorithm~3 allows
us to use a fixed stepsize $\la \in (0, \frac{\a(2\th -1)}{\th
L})$. In this case, taking $\t= \t_n = \frac{\th}{2\th -1}$, steps 1
and 2 of Algorithm~3 can be written as
\begin{align*}
        y_n &= x_n + \frac{\th}{2\th-1} (x_n - x_{n-1})\\
        x_{n+1} &= \prox_{\la g}(x_n - \la \nabla f(y_n)).
\end{align*}
If $\th = 1$ this scheme reduces to the basic
reflected proximal gradient method.

Using Lemma~\ref{basic_lemma3} we can derive the same ergodic rate of
convergence of Algorithm~3 as in Section~\ref{sec:rate}.

\section{Comparison} \label{sec:comparison}
In case of \eqref{min} under the assumption that $\nabla f$
is Lipschitz-continuous there are many possible linesearch rules
for proximal methods. One of the most simple is
Goldshtein-Armijo-like procedure
proposed in \cite{fista}. The proximal gradient method with this
backtracking rule 
generates $(x_n)$ by the following scheme
\begin{flushleft}
    for given $x_{n}$, $\la_{n-1}$ take $\b \in(0,1)$, $\la = \la_{n-1}$ and run\\
    \texttt{repeat}\\
    \quad $z = \prox_{\la g} (x_n -\lambda\nf{x_n})$\\
    \quad \texttt{break if}\\
    \qquad\qquad $f(z) \leq f(x_n) + \lr{\nf{x_n}, z-x_n} + \frac{1}{2\la} \n{x_n-z}^2$\\
    \quad    \texttt{update} $\la:=\b \la$\\
    \texttt{return $\la_n:= \la$, $x_{n+1}:= z$}
\end{flushleft}
Each iteration of such backtracking requires computation of $f$ and
$\prox_{\la g}$. And even if the linesearch terminates in one
iteration, we have to compute $f(z)$ and $f(x_n)$ in order to make
sure that the stopping criteria of the linesearch is satisfied.
Moreover, the sequence of stepsizes $(\la_n)$ must be nonincreasing.

Although we are aware that there are several methods \cite{nesterov07,scheinberg14,birgin-raydan,bb-gradient} for optimization
problems that allow to use nonmonotone steps, we do not consider
them.  In any case one can hardly cover all the methods in one paper, so we have chosen only the few most widespread
methods. Another reason is that we want to emphasize the importance of the algorithms that
use nonmonotone steps. A detailed comparison of our methods with other
optimization methods remains for future research.



For a general problem \eqref{gvi} one can apply the  forward-backward-forward method
proposed by Tseng~\cite{tseng00}. It generates the sequence $(x_n)$ by the
following rule
\begin{flushleft}
    for given $x_{n}$, $\la_{n-1}$, $\th \in (0,1)$ take $\b
    \in(0,1)$, $\la = \delta \la_{n-1}$ and run\\
    \texttt{repeat}\\
    \quad $z = \prox_{\la g} (x_n -\lambda F(x_n)) $\\
    \quad \texttt{break if}\\
    \qquad\qquad $\la \n{F(z) - F(x_n)} \leq \th \n{z  - x_n}$\\
    \quad    \texttt{update} $\la:= \b \la$\\
    \texttt{return $\la_n:= \la$, $z_{n}:= z$}\\
    Compute $x_{n+1}  = z_n - \la_n (F(z_n)-F(x_n))$.
\end{flushleft}
The choice of value $\delta$ is quite important. Originally in the
paper $\delta = 1$. However, this exclude possibility to
enlarge stepsizes. As heuristic we propose to use $\delta > 1$ and instead
control boundedness of $(\la_n)$.

Evidently, the stopping criteria of the linesearch in Tseng's method is very similar to
\eqref{lipsch_main2}. However, each iteration of the former requires
evaluation of $z$.  In the same time, the Tseng's method is more
general, as it allows to solve more general problems and requires only
continuity of $F$.

Recently, there appeared paper \cite{cruz15} in which the authors
applied Tseng's method for problem \eqref{min}. Using the specificy of
the problem, they proposed novel linesearch procedures and
obtained the complexity results for their methods. One of such method
is the following
\begin{flushleft}
    for given $x_{n}$,  $\th \in (0,1)$  compute \\
        \quad $z_n = \prox_{g} (x_n - \nabla f(x_n)) $ and run\\
        \texttt{repeat}\\
       $x_{n+1}  = x_n - \b_n (x_n - z_n)$\\
    \quad \texttt{break if}\\
    \qquad\qquad  $(f+g)(x_{n+1}) \leq (f+g)(x_n) - \b(g(x_n)-g(z_n)) - \b
    \lr{\nf{x_n},x_n-z_n}+\frac{\b}{2}\n{x_n-z_n}^2$ \\ 
    \quad    \texttt{update} $\b:= \th \b $\\
    \texttt{return $x_{n+1}$}
\end{flushleft}
As one can see, each iteration of this linesearch needs only a new
value $(f+g)(x_{n+1})$ and simple vector-vector computation. However, the main drawback is that it uses
quite conservative stepsizes. Because of this, we did not include this
method in our numerical experiments. But in any case this direction
seems to be very interesting.

\subsection{Numerical illustration}\label{numeric}
Our test problems include four optimization problems, saddle point
problem and one nonlinear variational inequality.  For the
optimization problems we present a comparison of all our algorithms
with PGM (projected gradient method with linesearch from Section
\ref{sec:comparison}), FISTA (accelerated proximal method with the
same linesearch), and FBF (Tseng's forward-backward-forward method as
described in Section~\ref{sec:comparison}).  For a variational
inequality we ran two variants of FBF with $\d=1$ and $\d=2$. For a
saddle point problem we additionally included into the comparison the primal-dual method
of Chambolle and Pock~\cite{chambolle2011first}. Computations\footnote{All codes can be found on \url{https://gitlab.icg.tugraz.at/malitsky/pegm}} were
performed using Python 2.7 on an Intel Core i3-2348M 2.3GHz running
64-bit Linux Mint 17.

For each problem we present plots (residuals vs iterations) and also
give numerical illustration of the efficiency of the algorithms.

The parameters were chosen as follows
\begin{itemize}
    \item Alg.1, Alg.2:  $\a = 0.41$, $\sigma = 0.7$;
    \item Alg.3: $\a = 0.41$, $\th = 2$, $\sigma = 0.7$;
    \item  PGM and FISTA:  $\b = 0.7$, $\la_0 = 1$;
    \item FBF: $\b = 0.7$, $\th = 0.9$.
\end{itemize} 
We did not set $\la_{max}$ for our methods, since it is rather a
theoretical requirement.  For our methods as well as for FBF we used
the initialization procedure as described in Remark~\ref{rem:init}. Also note that $\sigma$
in our methods and $\b$ in FBF, PGM, and FISTA play the same roles,
that is why we choose them equal. The initial stepsize $\la_0 = 1$ for
FISTA and PGM was chosen larger that it was predicted by the
linesearch.

In many examples below we used a random generated data. Usually we ran several
experiments with the same distribution and if there was no large
discrepancy, we chose one sample from these experiments for the
presentation. Also for some of the problems we intentionally took
starting points that were quite far from a solution in order to make a
problem harder. 
   
 \subsection*{Constrained minimization}
Consider the following  minimization problem
\begin{equation}
    \label{cons_min:1}
    \min_{x\in C} f(x) = \sum_{i=1}^{d}q_i (e^{x_i} - x_i - 1) + \frac 1 2 \n{x}^2,
\end{equation}
where $x \in \R^d$, $q \in \R^d_+$ and $C = \{z\in \R^d \colon
\n{z}\leq 100\}$. Clearly, problem~\eqref{cons_min:1} is an instance
of \eqref{min} with $g(x) = \d_C(x)$.  Since the set $C$ is compact,
$\nabla f$ is Lipschitz-continuous on it. However, $\nabla f$ changes
quite fast and hence, our method is in the advantageous
situation. Note that $f$ is strongly convex.  We took $d = 10$, and
generated $q$ uniformly randomly from $(0,1000)^d$. The starting point
 $x_0$ was chosen uniformly randomly from $(-50,50)^d$.  The results
 are presented on Fig.~\ref{fig:cons_geom} and in
 Table~\ref{tab:cons_geom}.
    
\begin{figure}
    \centering
    \begin{subfigure}[b]{0.4\textwidth}
        \includegraphics[width=\textwidth]{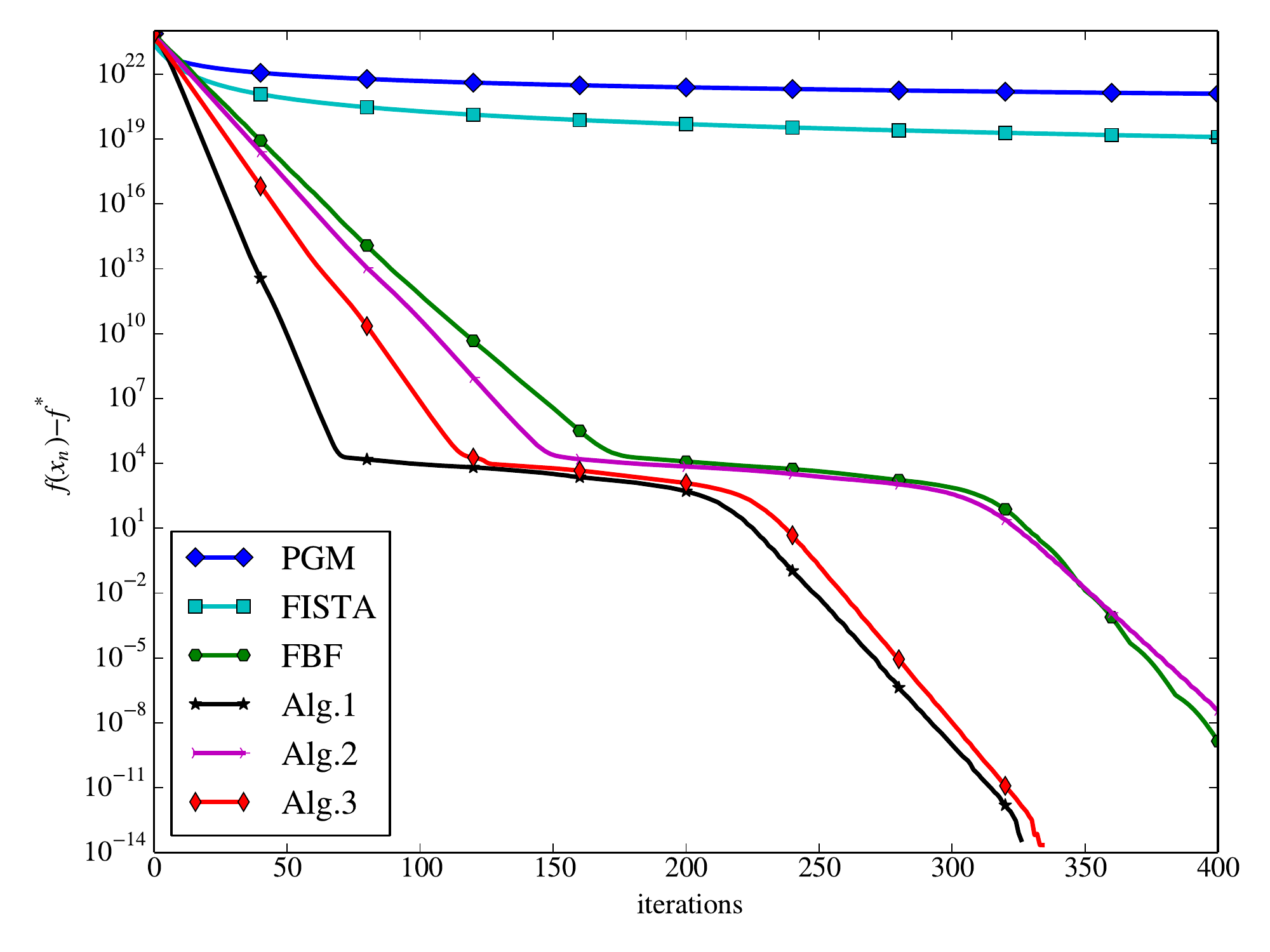}
    \end{subfigure}
        \centering
    \begin{subfigure}[b]{0.4\textwidth}
        \includegraphics[width=\textwidth]{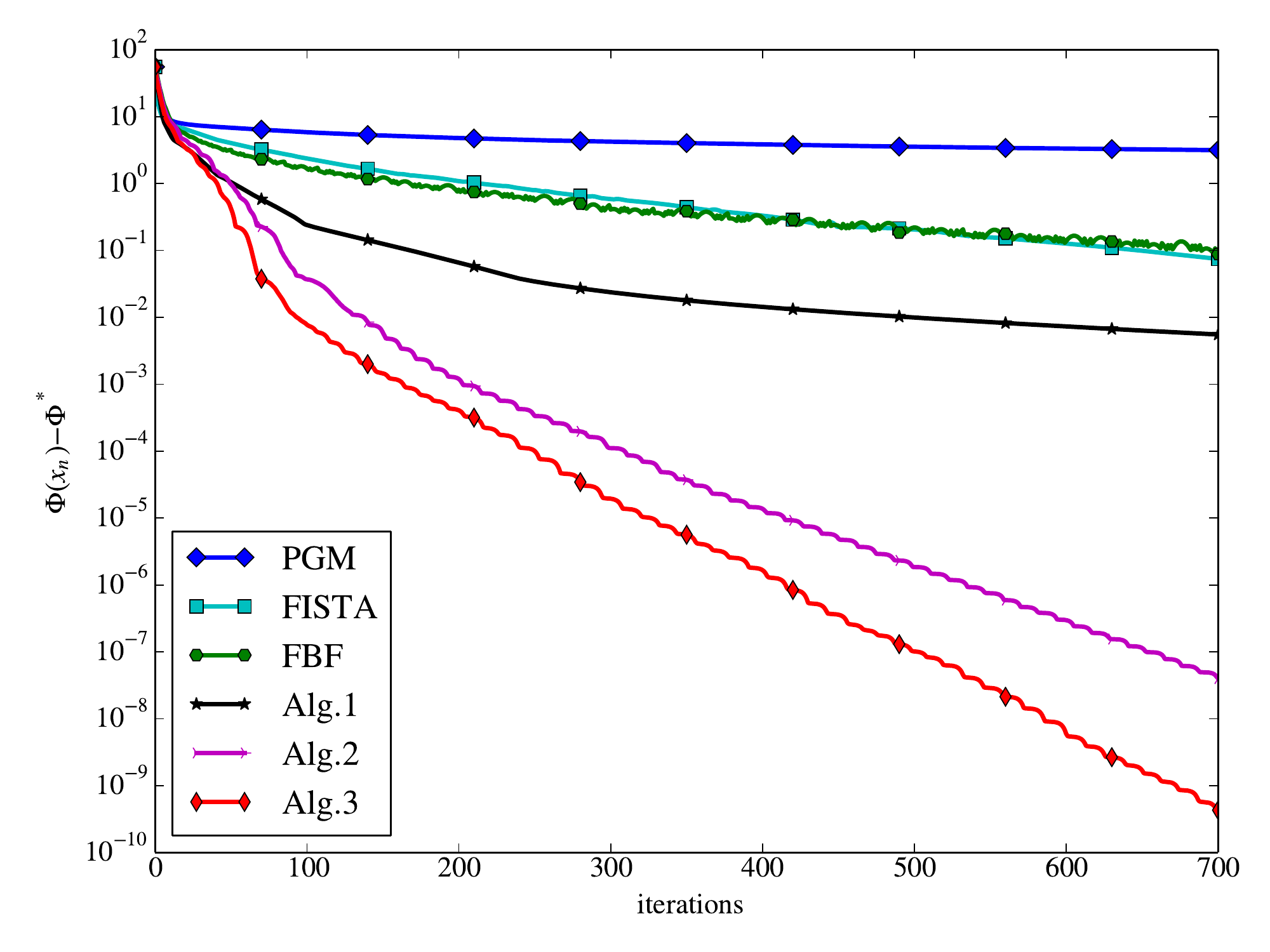}
    \end{subfigure}
        \caption{Results for problems \eqref{cons_min:1}
        (left) and \eqref{geo:1}  (right)}
        \label{fig:cons_geom}
\end{figure}
 
\begin{table}\centering
    \footnotesize
\begin{tabular}{|l|rrrrr||rrrrr|}
    \hline
    & \multicolumn{5}{c||}{Constrained minimization
    \eqref{cons_min:1}} & \multicolumn{5}{c|}{Geometric programming \eqref{geo:1}}  \\
    &  \#iter  & \#$f$  & \#grad  &  \#prox  &   time   &  \#iter  & \#$f$  & \#grad  &  \#prox  &   time \\ \hline                                                          
    PGM     &     400  &  553  &       553  &     553  &  0.04  &     700  &   720  &       720  &     720  &  0.14  \\
    FISTA   &     400  &  953  &       553  &     553  &  0.06  &     700  &  1420  &       720  &     720  &  0.14 \\
    FBF     &     400  &    0  &      1446  &    1045  &  0.06  &     700  &     0  &      2746  &    2045  &  0.20 \\
    Alg.1  &     400  &    0  &       608  &     400  &  0.04   &     700  &     0  &       708  &     700  &  0.10\\
    Alg.2  &     400  &    0  &       700  &     400  &  0.04   &     700  &     0  &      1472  &     700  &  0.13\\
    Alg.3  &     400  &    0  &       626  &     400  &  0.04   &     700    &     0  &      1293  &     700  &  0.13\\
 \hline
\end{tabular}\caption{Results for problems \eqref{cons_min:1} and
\eqref{geo:1}}  \label{tab:cons_geom}
\end{table}

In fact, for this particular problem FBF and our proposed methods are almost equal regarding a
speed of convergence. With different input data each
of the fore-mentioned algorithms might show the best performance. However,
 our algorithms require much less evaluation.

\subsection*{Geometric programming} We consider a canonical example of geometric programming
\cite{boyd:convex_optim} for which we add $l_1$--norm:
\begin{equation}
    \label{geo:1}
    \min_{x\in \R^d} \Phi(x):=\sum_{i=1}^m e^{\lr{a_i,x}+b_i} + \lr{c,x} + \n{x}_1,
\end{equation}
where $a_i,c \in \R^d$, $b\in \R^m$. Obviously, \eqref{geo:1} is a particular case of \eqref{min} with
$$ f(x) =\sum_{i=1}^m e^{\lr{a_i,x}+b_i} + \lr{c,x}, \quad g(x) = \n{x}_1.$$
Clearly, $\nabla f(x)$ is not Lipschitz-continuous.  We took $d = 100$, $m = 50$ and generated data
$a_i$, $b$ and $c$ uniformly randomly from $(0,1)^d$, $(-1,1)^m$, and $(-1,1)^d$ respectively. The starting
point was chosen as $x_0 = (0,\dots, 0)$.  The results are presented
on Fig.~\ref{fig:cons_geom} and in Table~\ref{tab:cons_geom}.
  
Alg.1 shows the worst performance among proposed methods. In fact, there is
no theoretical  guarantee for its convergence for this problem. FBF with $\d=2$ behaves similarly to FISTA and
requires too much evaluation in contrast with our methods. However,
FBF with $\d = 1$ behaved even worse as it almost coincided with PGM.
  
\subsection*{Analytic center}  
Suppose that set $C$ is a solution set of the following system of convex inequalities
$$f_i(x)\l 0, \quad i = 1,\dots, m. $$
The analytic center of the $C$ is defined as an optimal point of the problem
\begin{equation}
    \label{ac:1}
    \min_{x\in \R^d}  \sum_{i=1}^m -\log (-f_i(x))
\end{equation}
This is a convex unconstrained minimization problem that is an instance of~\eqref{min} with $g(x) =
0$. However, in general $\nabla f$ is not Lipschitz-continuous.

In our experiment we seek the analytic center of $C$ that is defined
as a polyhedron. In other words, we set $f_i(x) = \lr{a_i,x}-b_i$ for
$i = 1,\dots, m$, where $a_i\in \R^d$, $b_i\in \R$.  For a particular
example we took $d = 100$, $m = 1000$, and generated $a_i$ uniformly randomly
from $[-1,1]^d$. First $100$ coordinates of $b$ we set to $0.01$, the
rest to $100$. On the one hand, this guarantees that $x_0 = (0,\dots,
0)$ belongs to $C$. And on the other hand, this makes $x_0$ close to some
 vertex of $C$ and hence, probably far from the analytic
 center. Because of this choice, $\nabla f$ changes very fast and PGM
 and FISTA do not give satisfactory results. The results are presented
on Fig.~\ref{fig:ac_l_p} and in Table~\ref{tab:ac_l_p}. Since this
is an unconstrained problem, we ran Alg.1 using Remark~\ref{rem:C_affine}.

Similarly to the previous problem, FBF with $\d=2$ requires too much
evaluation but with $\d=1$ its performance was quite poor (the same as
 FISTA and PGM).

\begin{figure}
    \centering
    \begin{subfigure}[b]{0.4\textwidth}
       \includegraphics[width=\textwidth]{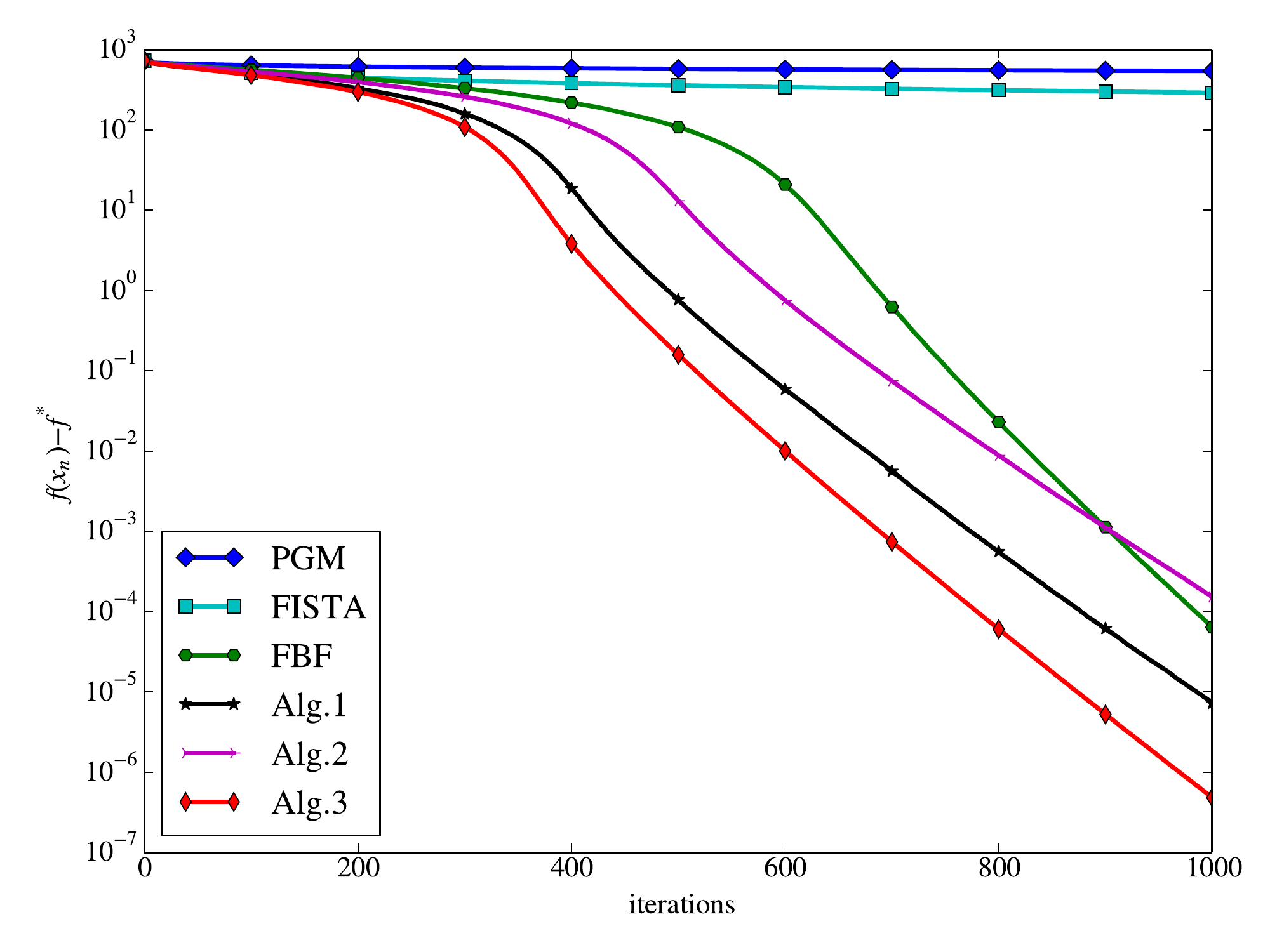}
    \end{subfigure}
        \centering
    \begin{subfigure}[b]{0.4\textwidth}
        \includegraphics[width=\textwidth]{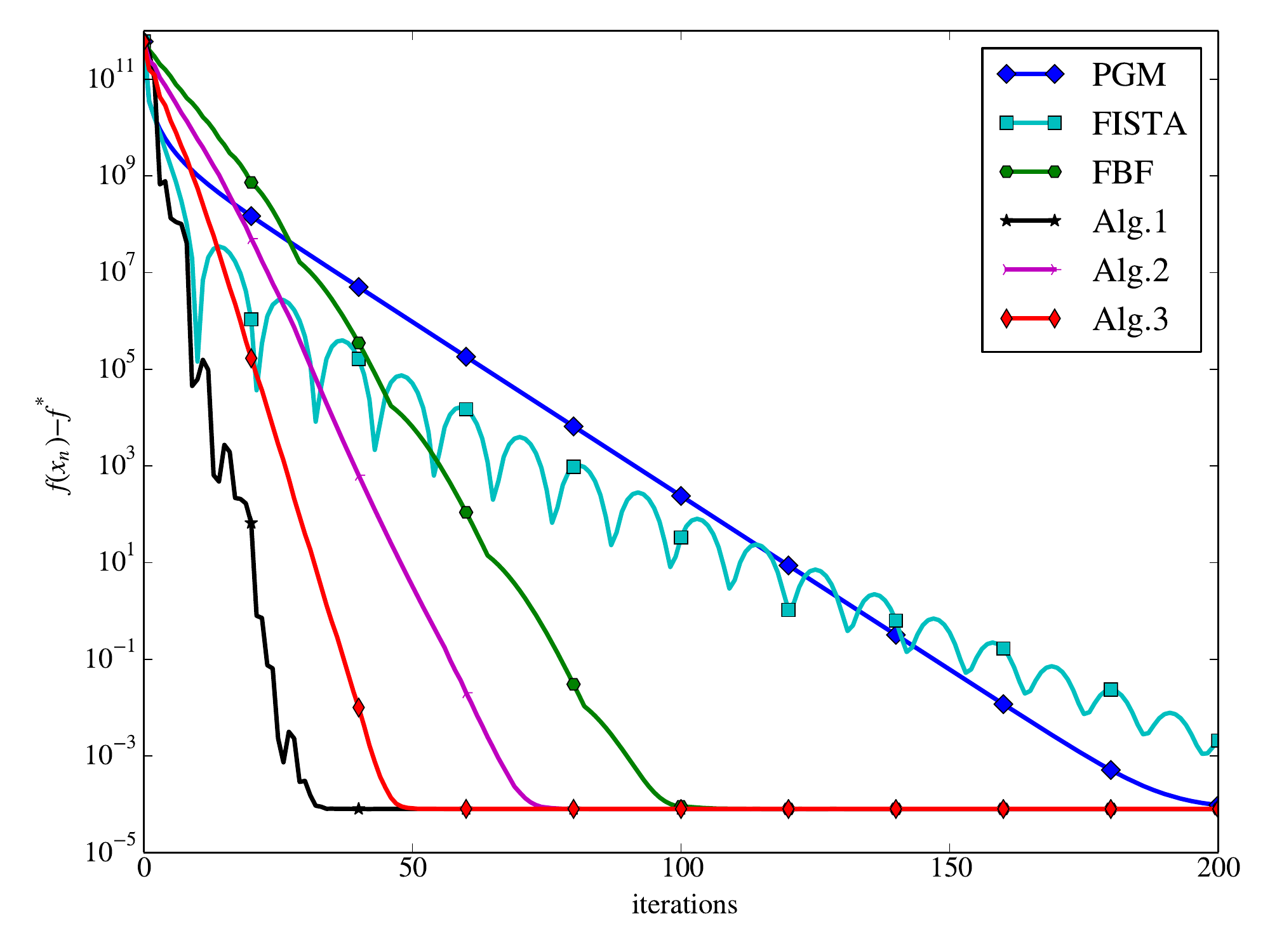}        
    \end{subfigure} 
    \caption{Results for problems \eqref{ac:1} (left) and 
   \eqref{eq:l_p} (right)}\label{fig:ac_l_p}
\end{figure}

\begin{table}\centering
    \footnotesize
\begin{tabular}{|l|rrrrr||rrrrr|}
    \hline
    & \multicolumn{5}{c||}{Analytic center \eqref{ac:1}} & \multicolumn{5}{c|}{$l_p$--minimization \eqref{eq:l_p} }  \\
         &  \#iter  & \#$f$  & \#grad  &  \#prox  &   time   &  \#iter  & \#$f$  & \#grad  &  \#prox  &   time \\ \hline                                                          
 PGM     &    1000  &  1044  &      1044  &    1044  &  0.96  &     200  &  235  &       235  &     235  &  0.04  \\ 
 FISTA   &    1000  &  2044  &      1044  &    1044  &  1.31  &     200  &  435  &       235  &     235  &  0.06  \\ 
 FBF     &    1000  &     0  &      3908  &    2907  &  1.83  &     200  &    0  &       784  &     583  &  0.06  \\ 
 Alg. 1  &    1000  &     0  &      1456  &    1000  &  0.89  &     200  &    0  &       312  &     200  &  0.04  \\ 
 Alg. 2  &    1000  &     0  &      1968  &    1000  &  1.11  &     200  &    0  &       405  &     200  &  0.04  \\ 
 Alg. 3  &    1000  &     0  &      1769  &    1000  &  1.08  &     200  &    0  &       369  &     200  &  0.04  \\ 
\hline
\end{tabular}\caption{Results for problems \eqref{ac:1}  and \eqref{eq:l_p}}  \label{tab:ac_l_p}
\end{table}

\subsection*{$l_p$--minimization} Consider a problem of
$l_p$--minimization
\begin{equation}
    \label{eq:l_p}
    \min_{x \in \R^d} f(x) = \frac 1 p \sum_{i=1}^m \n{x-a_i}^p,
\end{equation}
where $a_i \in \R^d$. It is clear that for $p\geq 2$ \eqref{eq:l_p} is
an instance of \eqref{min}. For a particular case $p = 2$ this a
well-known Fermat-Weber problem. We are interested in case when $p>2$
because this choice makes $\nabla f$ nonlinear. For a numerical
experiment we choose $d =50$, $m = 50$, $p = 3$ and generated points $a_i$
uniformly randomly from $[-100, 100]^d$. the starting points $x_0$ was
chosen uniformly randomly from $[-1000,1000]^d$. Although it is clear
that a solution of \eqref{eq:l_p} belongs to the convex hull of
$(a_i)$, we intentionally choose $x_0$ as a random point from the larger
range in order to make the problem harder. The results are presented
on Fig.~\ref{fig:ac_l_p} and in Table~\ref{tab:ac_l_p}. As previously,
this is an unconstrained problem, so we ran Alg.1 with Remark~\ref{rem:C_affine}.

\subsection*{Sun's problem} Now consider a variational inequality that
is not an instance of the optimization problem. We study a nonlinear VI, proposed by
Sun \cite{sun:94} 
\begin{equation}
    \label{eq:sun}
    \lr{F(x^*), x-x^*} \geq 0 \quad \forall x \in C,
\end{equation}
 where
 \begin{align*}
    F(x)   & = F_1(x) + F_2(x),\\
    F_1(x) & = (f_1(x),f_2(x),\dots,   f_d(x)),\\
    F_2(x) & = Dx+c, \\
    f_i(x) & = x_{i-1}^2 + x_i^2 +  x_{i-1}x_i + x_i x_{i+1},\quad   i=1,2,\dots, d,\\
    x_0 & = x_{d+1} = 0,
  \end{align*}
  Here $D$ is a square matrix $d\times d$ defined by condition
$$d_{ij}=
\begin{cases}
  4, & i = j,\\
  1, & i-j=1,\\
  -2,& i-j = -1,\\
  0, & \text{otherwise},
\end{cases}
$$
and $c=(-1,-1,\dots, -1)$. We defined the feasible set as
$C = [0,100]^d$. In the experiment we took $d = 1000$ and the starting
point was chosen uniformly randomly from $C$. We ran two variants of
Tseng's method with $\d = 1$ (TBF-1) and $\d = 2$ (TBF-2).  For the
comparison we used the residual $||x_n - P_C(x_n - F(x_n))||$.  The
results are presented on Fig.~\ref{fig:sun} and in Table~\ref{tab:sun}.
\begin{figure}
    \centering
    \includegraphics[width=0.4\textwidth]{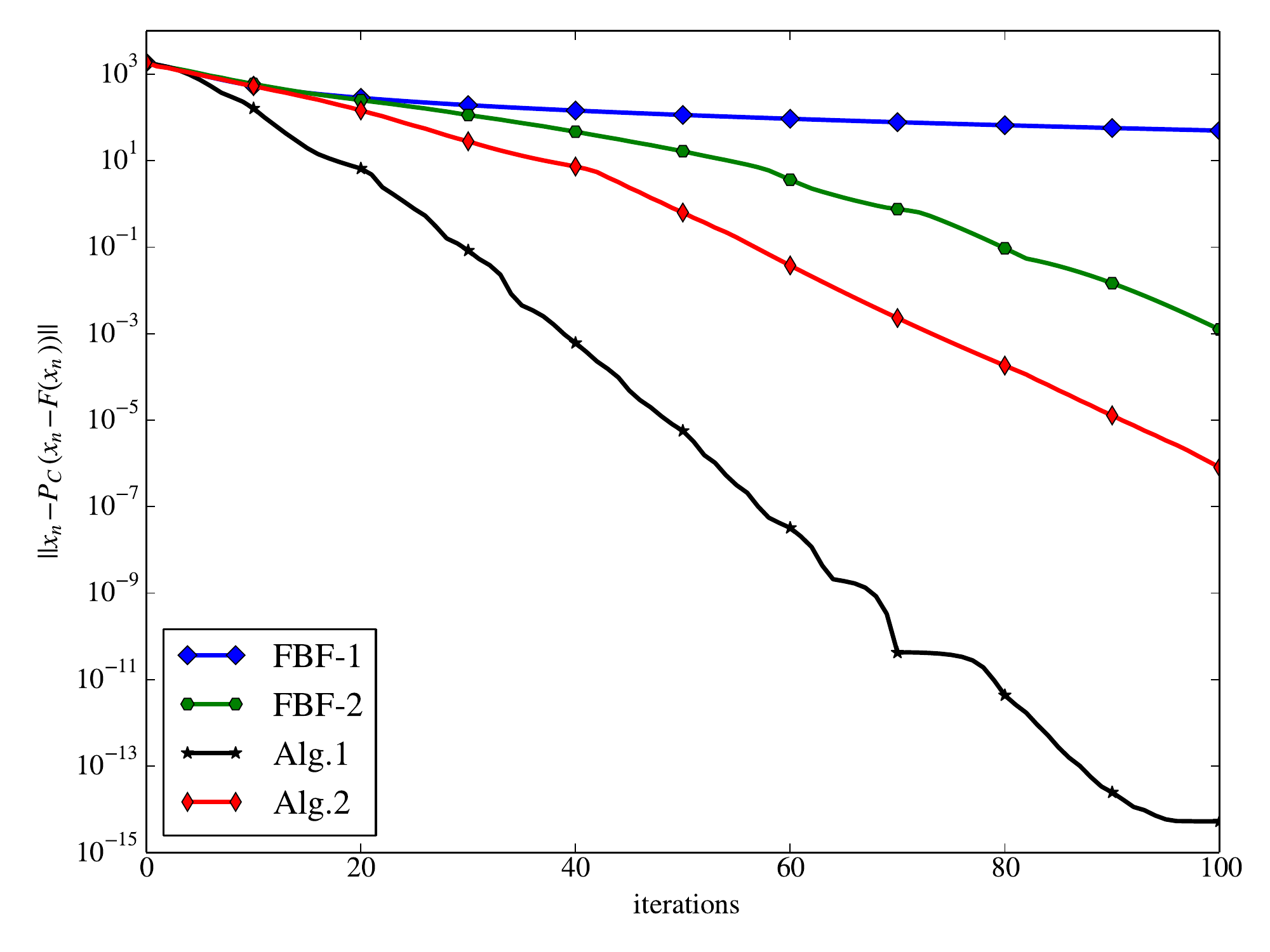}
        \caption{Results for problem~\eqref{eq:sun}
}\label{fig:sun}
\end{figure}

\begin{table}\centering
    \footnotesize
\begin{tabular}{|l|rrrr|}
\hline
         &  \#iter  &  \# $F$  &  \#prox  &  time  \\
\hline                                                 
 FBF-1   &     100  &       201  &     100  & 30.6  \\
 FBF-2   &     100  &       384  &     283  & 52.8  \\
 Alg. 1  &     100  &       228  &     100  & 37.7  \\
 Alg. 2  &     100  &       191  &     100  & 30.9  \\
 \hline
\end{tabular}
\caption{Results for problem~\eqref{eq:sun}}    \label{tab:sun}
\end{table}

\subsection*{Matrix game} We are interested in the following min-max
matrix game
\begin{equation}
    \label{eq:minmax}
    \min_{x \in \D_l}\max_{y\in \D_k} \lr{Ax, y},
\end{equation}
where $x\in \R^l$, $y\in \R^k$, $A\in \R^{k\times l}$, and $\Delta_k$,
$\D_l$ denote the standard unit simplices in $\R^k$ and $\R^l$
respectively. 
The problem \eqref{eq:minmax} is equivalent to the following
variational inequality
$$\lr{F(z^*), z-z^* } + g(z)-g(z^*) \geq 0$$
with $$z = \binom{x}{y}, \quad F(z) = \binom{A^T y}{-Ax},\quad g(z) = \binom{P_{\D_l}x}{P_{\D_k}y}.$$
As one can see, operator $F$ is linear, so we can run our methods
Alg.1 and Alg.2 using Remark~\ref{rem:F_affine}. 
In addition to FBF, we compared our methods
with the primal-dual algorithm of Chambolle and Pock (PD). For that method
we used fixed stepsizes $\tau = \sigma = \frac{1}{||A||}$ (as in paper \cite{pock:ergodic}). 
In our experiment we took $k = 1000$, $l = 2000$ and generated two
instances of matrix $A$ with entries (a) uniformly distributed  and
(b) normally distributed in $[-1,1]$.

The starting point for both cases was chosen as $x_0 =
\frac{1}{k}(1,\dots, 1)$ and $y_0 = \frac 1 l (1,\dots, 1)$. In order to compute projection onto the unit
simplex we used the algorithm from \cite{duchi2008efficient}.
For a comparison we used a primal-dual gap $\mathcal G(x,y)$ which can
be easily 
computed for a feasible pair $(x,y)$
$$\mathcal G(x,y) = \max_i (Ax)_i - \min_j (A^Ty)_j.$$
Since iterates obtained by Tseng's method may be infeasible,
for a computation of primal-dual gap in this case we used the auxiliary point
$z$ that is obtained by linesearch (see Section~\ref{sec:comparison}). As in the previous
example,  we ran two
variants of FBF with $\delta = 1$ and $\delta = 2$. For this problem,
instead of  $\# F$, we counted the number of matrix-vector
multiplication $\# \text{mult}$. Two projections onto simplices $\D_k$ and $\D_l$
respectively we counted as one $\text{prox}$.

When entries of $A$ were uniformly distributed, 
all algorithms behaved almost equally. Regarding the cost of
one iteration, FBF methods are more expensive than other
algorithms. Although the performance of PD, Alg.1, and Alg.2 was the
same, the former extremely depends on the value of $||A||_2$. It ran
much slower when we used $||A||_F$ instead of $||A||_2$. But
in case of a huge scale problem evaluating of $||A||_2$ will be also quite resource-intensive.

When entries of $A$ were normally distributed, Alg.2
showed much better performance. And in both cases our algorithms
required the same amount of computation as PD method. It is important
to note that
using FBF for this problem with $\d = 2$ makes the things only worse:
it gives almost nothing for the speed of convergence and instead
 uses much more computational resources. 
 \begin{figure}
    \centering
    \begin{subfigure}[b]{0.4\textwidth}
        \includegraphics[width=\textwidth]{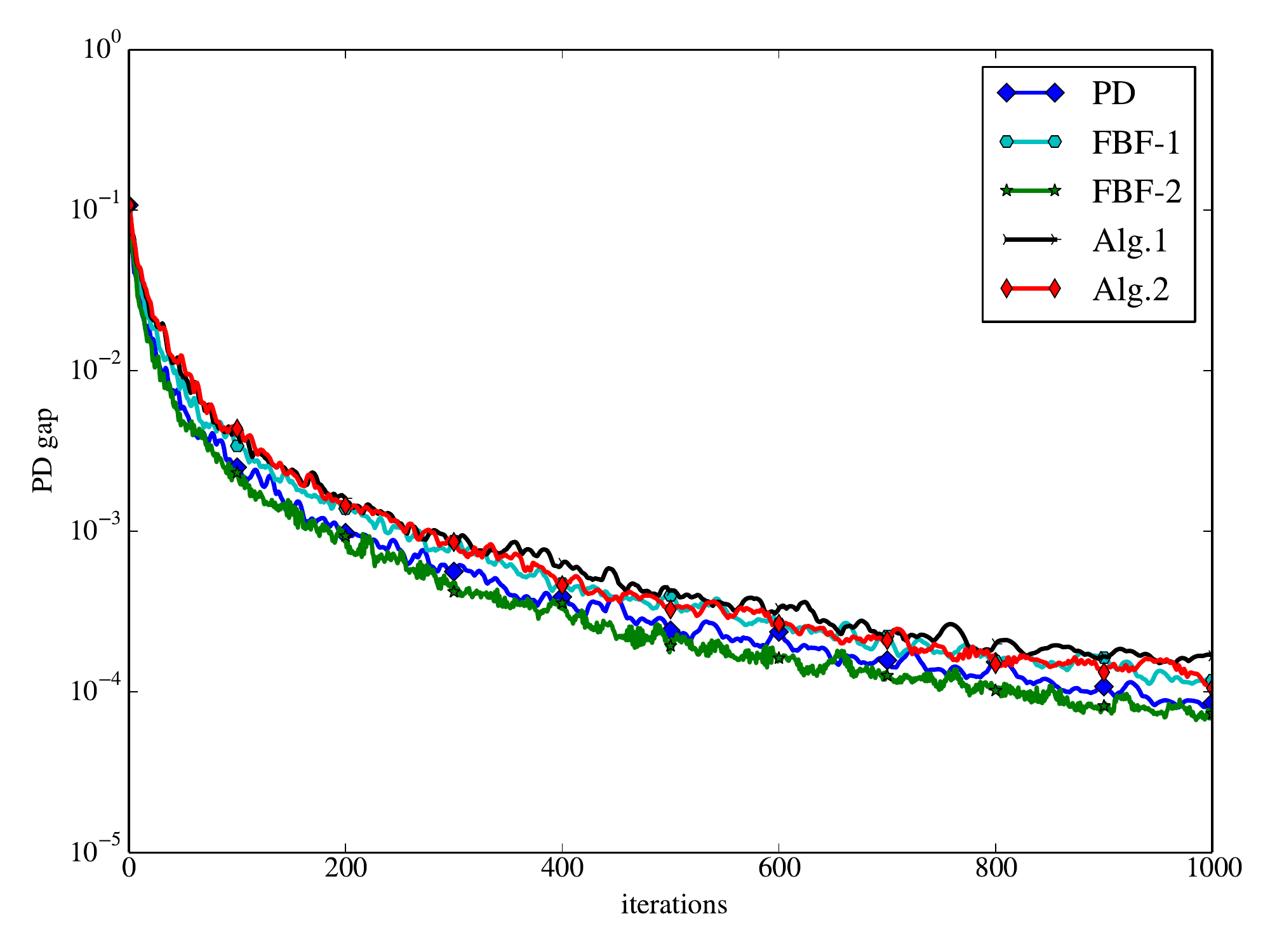}
    \end{subfigure}
        \centering
    \begin{subfigure}[b]{0.4\textwidth}
        \includegraphics[width=\textwidth]{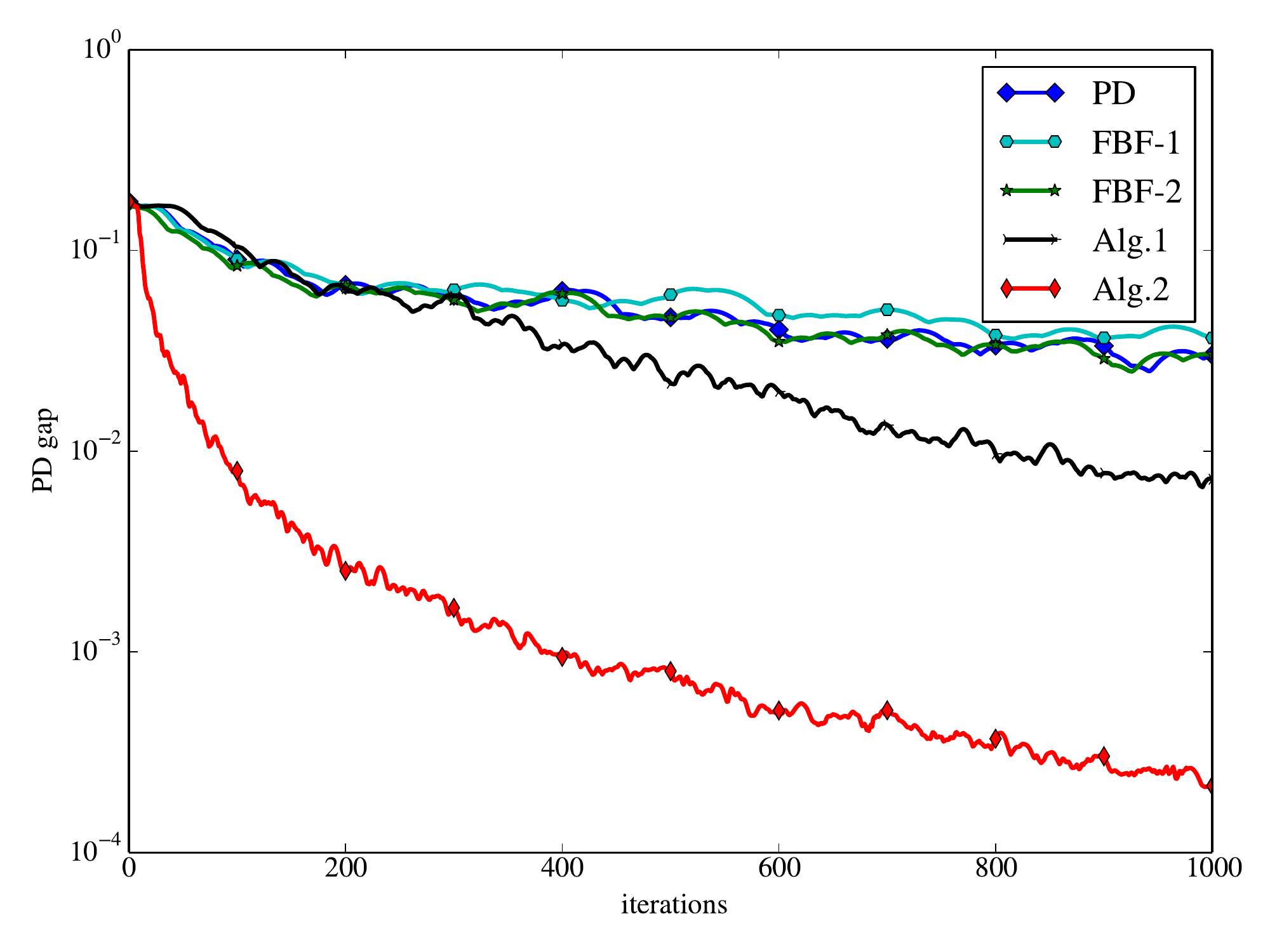}
    \end{subfigure}
        \caption{Results for problem \eqref{eq:minmax}: uniform case
        (left) and normal case
    (right)}\label{fig:minmax}
\end{figure}

\begin{table}\centering
    \footnotesize
\begin{tabular}{|l|rrrr||rrrr|}
    \hline
    & \multicolumn{4}{c||}{Uniform} & \multicolumn{4}{c|}{Normal}  \\
         &  \#iter  & \#mult   &  \#prox  &   time   &  \#iter  &   \#mult    &  \#prox  &   time \\ \hline
PD      &    1000  &    2000  &    1000  & 53.0  &    1000  &    2000  &    1000  &  53.3    \\ 
FBF-1   &    1000  &    4006  &    1002  & 77.3  &    1000  &    4004  &    1001  &  78.2  \\ 
FBF-2   &    1000  &    7892  &    2945  & 129.4 &    1000  &    7888  &    2943  &  128.4 \\ 
Alg. 1  &    1000  &    2004  &    1000  & 51.9  &    1000  &    2004  &    1000  &  51.8  \\ 
Alg. 2  &    1000  &    2004  &    1000  & 52.6  &    1000  &    2004        
&    1000  &    52.7  \\
\hline
\end{tabular}\caption{Results for problems \eqref{eq:minmax}}  \label{tab:3}
\end{table}

\section{Conclusion}
In this paper there were proposed several algorithms for a general
monotone variational inequality and a composite minimization
problem. All methods use some simple linesearch procedure that allow
to incorporate a local information of the operator. For all methods
there was established the ergodic rate of convergence. Numerical
experiments also approved their efficiency. Namely, we show that our
methods can outperform proximal gradient method or FISTA with standard
linesearch. Quite interesting is that the proposed methods become
extremely simple when the operator is affine. In particular, for a
minimax problem this makes the cost of one iteration the same as in
the primal-dual algorithm and, on the other hand, may allow to use
larger stepsizes. The requirement only of local Lipschitz-continuity
of the operator makes our methods very general.

As numerical simulations showed, the ratio $\# F /\# \text{iter}$
(number of evaluation of the operator to the number of iterations) is
almost always less than $2$. In the same time, this ratio for the
extragradient method or forward-backward-forward method equals $2$
even when they do not use any linesearch. Moreover, the ratio $\#
\text{prox} / \#\text{iter}$ for our methods always equals $1$.

The main drawback of the proposed methods is that we need the bound
$\a = \sqrt 2 -1$. This multiplier makes the steps smaller in case
when the Lipschitz constant of the operator does not change too
much. It is interesting to study whether this bound can be increased.

\vspace{0.5cm}
{\noindent
{\bfseries Acknowledgement:} The author acknowledges support from the Austrian science fund (FWF) under the
project "Efficient Algorithms for Nonsmooth Optimization in Imaging" (EANOI) No. I1148.}

\bibliographystyle{siam}       
\bibliography{publication} 
\end{document}